\documentclass[autodetect-engine]{article}
\usepackage{design_ASC}
\usepackage{bm}

\title{Parallel Surfaces of Cuspidal Cross Caps} 

\author{Atsuki Hiramatsu \\ 
    \footnotesize
    Department of Mathematics, Saitama University\\
    \footnotesize
    255 Shimo-Okubo, Sakura-ku, Saitama 338-8570, Japan\\
    \footnotesize
    E-mail addresses: \texttt{a.hiramatsu.454@ms.saitama-u.ac.jp}
}

\date{\today} 

\newcommand{\Nor}{\bm{\nu}}
\newcommand{\ps}{f^\varepsilon}
\newcommand{\PS}{$M^\varepsilon$\ }
\newcommand{\Lam}{\lambda^\varepsilon}
\newcommand{\Null}{\eta^\varepsilon}
\newcommand{\Locus}{\gamma^\varepsilon}

\numberwithin{equation}{section}

\begin{document}
\setlength{\droptitle}{-5em}    
\maketitle

\begin{abstract}
    This paper investigates the geometry and singularities of parallel surfaces of cuspidal cross caps, the fundamental non-front frontal singularities. We establish a criterion for the degeneracy of the distance squared function in terms of known geometric invariants and describe the resulting configuration of singularities. Our main result demonstrates that while the parallel surface is generically $\mathcal{A}$-equivalent to a cuspidal cross cap, it degenerates into a degenerated cuspidal $S_1$ singularity at specific distances characterized by the equation $C_2(\varepsilon)=0$. These distances act as a novel analogue of the principal radii of curvature. Indeed, although the Gaussian and mean curvatures diverge at the singularity, their asymptotic expansions reveal that their constant terms correspond to the product and average of the reciprocals of these distances, respectively. 
    
\end{abstract}


\section{Introduction}
    Parallel surfaces constitute a classical subject lying at the interface between differential geometry and singularity theory.  
    For a regular surface \(f:(\mathbb R^2,0)\to(\mathbb R^3,0)\) with unit normal vector field \(\bm{\nu}\), the parallel surface at distance \(\varepsilon\) is defined by
    \[
    f^\varepsilon(u,v)=f(u,v)+\varepsilon\bm{\nu}(u,v).
    \]
    It is well known that singularities naturally appear on parallel surfaces at principal radii of curvature. Generic singularities such as cuspidal edges and swallowtails occur on parallels of regular surfaces (see Figure \ref{fig:ExpPS}.), and their geometry has been extensively studied from both differential-geometric and singularity-theoretic viewpoints.
    \begin{figure}[h!]
        \centering
        \includegraphics[width=0.25\linewidth]{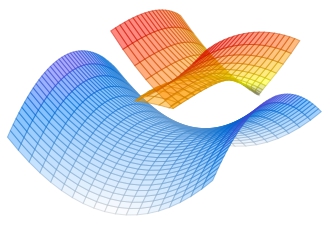}
        \caption{The example of a parallel surface of a blue saddle. We can see cuspidal edges in the orange parallel surface, and there is a swallowtail in the cuspidal edge.}
        \label{fig:ExpPS}
    \end{figure}
    
    A powerful approach to parallel surfaces was introduced by \cite{P} via the distance squared unfolding
    \[
    \Phi:\mathbb R^2\times \mathbb R^3\times \mathbb R\to\mathbb R,\qquad
    \Phi(u,v,x,y,z,\varepsilon)
    =
    \|f(u,v)-(x,y,z)\|^2-\varepsilon^2 .
    \]
    By Huygens' principle, parallel surfaces are realized as discriminant sets of \(\Phi\). Porteous showed that the singularities of \(\Phi\) reflect classical differential-geometric notions: normal directions correspond to \(A_1\)-singularities, principal radii of curvature to \(A_2\)-singularities, ridge points to \(A_3\)-singularities, and umbilics to \(D_4\)-singularities or worse.  
    Furthermore, Fukui and Hasegawa \cite{FukuiHasegawaPS} determined conditions under which the distance squared unfolding of a regular surface is versal with respect to the parameters \((x,y,z,\varepsilon)\), and identified the corresponding singularities of parallel surfaces, including cuspidal edges, swallowtails, cuspidal lips, and degenerate singularities.
    
    Parallel surfaces of fronts can be investigated in essentially the same manner as those of regular surfaces, since the front structure is preserved under parallel displacement. Consequently, the singularities arising on such parallels are again front singularities.  
    In contrast, parallel surfaces of frontals which are not fronts have not been extensively investigated. This situation is particularly interesting from the viewpoint of singularity theory, since genuinely frontal singularities may appear which cannot be detected through the classical theory of regular or front surfaces.
    
    In this paper, we study parallel surfaces of cuspidal cross caps, which are the simplest non-front frontal singularities. Recall that a map germ \(f:(\mathbb R^2,0)\to(\mathbb R^3,0)\) is a frontal if there exists a unit normal vector field \(\bm{\nu}\), while it is a front if \((f,\bm{\nu})\) is an immersion.  
    Typical examples of non-front frontals are the cuspidal \(S_k\)-singularities
    \[
    (u,v)\mapsto \bigl(u,v^2,v^3(u^{k+1}\pm v^2)\bigr),
    \]
    among which the case \(k=0\) is called the cuspidal cross cap.
    \begin{figure}[h!]
        \centering
        \includegraphics[width=0.25\linewidth]{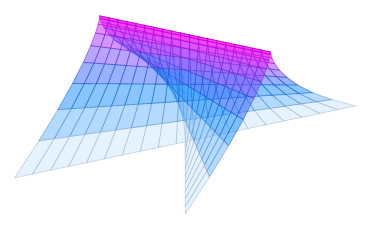}
        \caption{Cuspidal cross cap}
        \label{fig:CCCs}
    \end{figure}
    
    Unlike the regular case, the distance squared unfolding is never versal for cuspidal cross caps. Nevertheless, the degeneracy conditions of the distance squared unfolding still contain essential geometric information. The purpose of this paper is to clarify how such degeneracies control the singularities of parallel surfaces.
    
    Our main result describes when the singularity degenerates on parallels of cuspidal cross caps (Theorem \ref{thm:PSofCCC}). We show that the degeneration distances of parallel surfaces are characterized by the equation 
        $$C_2(\varepsilon):= \varepsilon^2(\tau_0-\theta_1)^2 - \qty(\frac{b_{04}}{3}\varepsilon + 1)(\varepsilon \kappa_0\sin\theta_0 + 1) =0,$$
    where the coefficients $\tau_i,\theta_i,\kappa_i,b_{ij}$ are geometric invariants defined by Section \ref{Sec:2}. Furthermore, give a necessary and sufficient condition for a cuspidal cross cap to degenerate into a degenerated cuspidal $S_1$ singularity (Corollary \ref{cor:PSofcS1}). Using the same calculation, we also show that parallels of cuspidal $S_1$ are $\mathcal{A}$-equivalent to cuspidal $S_1$ when $C_2(\varepsilon)\not=0$ (See Corollary \ref{cor:PSofcS1}.).
    
    We also express the degeneracy conditions of the distance squared unfolding (Lemma \ref{prop:DegConSumm}) in terms of the geometric invariants introduced in \cite{SUY,MS}.
    
    The distances at which parallel surfaces degenerate are naturally recognized as analogues of the principal radii of curvature for regular surfaces. In the case of cuspidal cross caps, although the Gaussian and mean curvatures diverge at the singularity, their asymptotic expansions with respect to the coordinates of Fukui's normal form reveal that the constant terms coincide with the product and average of the reciprocals of these distances, respectively (see Remark \ref{rmk:GaussMean}). Therefore, we regard these distances as a novel type of radii of curvature. Notably, these do not coincide with the radii obtained from the limits of conventional principal curvatures, both of which diverge.
    
    The paper is organized as follows. In Section \ref{Sec:2}, we review the Fukui's normal form for cuspidal edges and adapt this framework to the case of cuspidal cross caps. We then compute several geometric invariants, including the coefficients of the first fundamental form. In Section \ref{Sec:3}, we establish the criteria for the degeneracy of the distance squared unfolding and describe the resulting configuration of singularities. In Section \ref{Sec:4}, we show that a parallel surface of a cuspidal cross cap is, in general, $\mathcal{A}$-equivalent to a cuspidal cross cap. Furthermore, we prove that if the distance $\varepsilon$ satisfies $C_2(\varepsilon) = 0$, the parallel surface becomes $\mathcal{A}$-equivalent to the degenerated $S_1$ singularity introduced in Definition \ref{def:DegCS1}.


\section{Fukui's normal form} \label{Sec:2}
    \subsection{The definition of Fukui's normal form}
        In this section, we recall Fukui's normal form of a cuspidal edge following Fukui \cite{FukuiSW}. Surfaces with Fukui's normal form may have a cuspidal cross cap at the origin.
    
        Let $\gamma : \R, \bm{0} \rightarrow \R^3,\bm{0}$ be a regular curve with arc the length parameter $s$, and the frame $\{ \bm{t}, \bm{n}, \bm{b} \}$ denotes its Frenet-Serret frame. We consider a map-germ $f:\R^2, \bm{0} \rightarrow \R^3, \bm{0}$ representing a singular surface, which satisfies the following conditions:
        
        There is a sequence $\{\bm{f}_k:\R,\bm{0} \rightarrow \R^3,\bm{0}, s\mapsto \bm{f}_k(s) \}_{k=1,2,\dots}$ of $C^\infty$-maps such that
            \begin{itemize}
                \item[\ronumi] for sufficiently large $m \in \N$, the following holds:
                    $$f(s,t) = \gamma(s) + \sum_{k=1}^{m} \bm{f}_k(s) \frac{t^k}{k!} + O^{m+1}(t),\vspace{-0.5em}$$
                \item[\ronumii] the singular set of $f$ is given by $S(f) = \qty{t=0}$, \vspace{-0.5em}
                \item[\ronumiii] for $k=1,2,\dots$, the vector $\bm{f}_k(s)$ are orthogonal to $\gamma'(s)$, where $'$ denote derivative with resopect to the arc length parameter $s$, and \vspace{-0.5em}
                \item[\ronumiv] The quantity $t^2/2$ is an arc length parameter of the restriction of $f$ along $t$ in the $\bm{n}$, $\bm{b}$ plane, that is, 
                    $$\langles{f_t(s,t), f_t(s,t)} = t^2.$$
            \end{itemize}
        By condition \ronumiii{}, we have $\langles{f_t(s,0), f_t(s,0)} = 0$, and $f_t(s,0) = \bm{f}_1(s)$, then we conclude that $\bm{f}_1 = 0$. We remark that the a vector field $\eta = \partial_t$ on $S(f)$, i.e., $df(\eta) = 0$ on $S(f)$. 
    
        When the curvature $\kappa$ of $\gamma$ is not zero, we have the following Frenet-Serret formula for $\gamma$:
            $$\vecTri{\bm{t}'}{\bm{n}'}{\bm{b}'} = \qty(\begin{array}{ccc}
                 0 & \kappa & 0  \\
                 -\kappa & 0 & \tau \\
                 0 & \tau & 0
            \end{array}) \vecTri{\bm{t}}{\bm{n}}{\bm{b}}$$
        Let us define $\theta$ $(0 \leq \theta \leq \pi)$ by $\cos \theta = | \bm{t} \ \bm{b} \ \bm{f}_2|$. This angle $\theta(s)$ represents a deviation of the vector $\bm{f}_2(s)$ from the plane spanned by $\{ \bm{n}(s), \bm{b}(s) \}$. When we denote $\bm{a}_1 = \bm{t}$, $\bm{a}_2 = \cos\theta \bm{n} - \sin \theta \bm{b}$ and $\bm{a}_3 = \bm{a}_1\times\bm{a}_3 = \sin \theta \bm{n} + \cos \theta \bm{b}$, the surface $f(s,t)$ is rewritten as follows in a new orthogonal frame $\{ \bm{a}_1, \bm{a}_2, \bm{a}_3 \}$:
            \begin{equation} \label{eq:FukuiNormalForm}
                \begin{split}
                    f(s,t) &= \gamma(s) + \sum_{k=2}^m \qty{a_k(s) \frac{t^k}{k!}} \bm{a}_2(s) + \qty{\sum_{k=2}^m b_k(s) \frac{t^k}{k!}}\bm{a}_3(s) + O^{m+1}(t)\\
                    &= \gamma(s) + a(s,t)\bm{a}_2(s) + b(s,t)\bm{a}_3(s) + O^{m+1}(t),
                \end{split}
            \end{equation}
        where $a_k(s) = \langles{\bm{f}_k(s), \bm{a}_2(s)}$ and $b_k(s) = \langles{\bm{f}_k(s), \bm{a}_3(s)}$.
    
        \begin{lem}\label{lem:FukuiFormCoeff}[\cite{FukuiSW}, Lemma 1.4.]
            The coefficients $a_k$ ($k \geq 3$) are determined by the lower order terms of $b$ inductively. Precisely speaking, $a_k$ is determined by $b_2, b_3, \dots, b_{k-1}$. In particular, $a_2=1$, $b_2=0$, $a_3=0$, $a_4=-\frac{3}{4}b_3^2$, $a_5=-2b_3b_4$+.
        \end{lem}
    
        \begin{proof}
            From the condition \ronumi{}, we have
                $$t^2 = \langles{f_t, f_t} = \langles{\bm{f}_2, \bm{f}_2} t^2 + \langles{\bm{f}_2, \bm{f}_3} t^3 + \sum_{k=4}^m \qty{2\frac{\langles{\bm{f}_2, \bm{f}_k}}{(k-1)!} + \sum_{k=2}^{k-2} \frac{\langles{\bm{f}_{i+1}, \bm{f}_{k-i+1}}}{i!(k-i)!}}t^k.$$
            Since $\bm{a}_2=\bm{f}_2$ and $\bm{f}_k$ is written by $a_k \bm{a}_2 + b_k \bm{a}_3$, in particular $\bm{f}_2 = a_2 \bm{a}_2 + b_2 \bm{a}_3$ and $\bm{f}_3 = a_3 \bm{a}_2 + b_3 \bm{a}_3$, we obtain $a_2=1$, $b_2=0$, $a_3 = \langles{\bm{f}_3, \bm{a}_2} = \langles{\bm{f}_3, \bm{f}_2} = 0$ and $a_k = \langles{\bm{f}_k, \bm{f}_2}$. Therefore $a_k$ ($k\geq 3$) are determined inductively by $b_3, b_4, \dots b_{k-1}$, inductively.
        \end{proof}
    
        The functions $b_k(s)$ are known as invariants of cuspidal edges, under actions by orientation-preserving diffeomorphisms of the source that preserve the singular curves with its orientation, together with rotation of $\R^3$ (see \cite[Remark 1.5.]{FukuiSW}). In addition, $b_k$ appears in the criterion for a cuspidal edge, a cuspidal cross cap, ..., etc.
    
        \begin{fact}[\cite{FukuiSW}, Proposition 1.6.] \label{fact:INV}
            Let $f:\R^2,\bm{0} \rightarrow \R^3,\bm{0}$ be a map as in the first paragraph in this section. We have that \vspace{-0.5em}
            \begin{itemize}
                \item the singularity of $f$ is cuspidal edge if $b_3(0) \not = 0$, and \vspace{-0.5em}
                \item the singularity of $f$ is cuspidal cross cap if $b_3(0) = 0$, $b_3'(0) \not= 0$.
            \end{itemize}
        \end{fact}
        
        \begin{proof}
            See Appendix B in \cite{FukuiSW}.
        \end{proof}
    
        \begin{prop}\label{prop:DegCSk}
            Let $f:\R^2,\bm{0} \rightarrow \R^3,\bm{0}$ be a map as in the first paragraph in this section. The singularity of $f$ is the cuspidal $S_1$-singularity at $\bm{0}$ if $b_3(0) = b_3'(0) = 0$, $b_5(0)b_3^{(2)}(0) \not= 0$.
        \end{prop}
    
        \begin{proof}
            The proof is given later when we compute the Weingarten matrix in Fukui's form.
        \end{proof}

    \subsection{Taylor expansions of some notions in Fukui's normal form}
        \subsubsection{Taylor expansions of Fukui's frame and the singular locus.}
            Fukui's frame $\bm{a}_1,\bm{a}_2,\bm{a}_3$ holds a formula like Frenet-Seret Theorem.
            $\gamma$
            
            \begin{lem}\label{lem:SeFrType}
                Assume the curvature $\kappa$ of the singular locus $\gamma$ is not zero. We have
                    $$\vecTri{\bm{a}'_1}{\bm{a}'_2}{\bm{a}'_3} = \matrixTT{0}{\kappa\cos\theta}{\kappa\sin\theta}{-\kappa\cos\theta}{0}{\tau-\theta'}{-\kappa\sin\theta}{-(\tau-\theta')}{0}\vecTri{\bm{a}_1}{\bm{a}_2}{\bm{a}_3}$$
            \end{lem}
    
            \begin{proof}
                Since $\vecTri{\bm{a}_1}{\bm{a}_2}{\bm{a}_3} = \qty(\begin{array}{ccc}
                    1 & 0 & 0 \\
                    0 & \cos\theta & -\sin\theta \\
                    0 & \sin\theta & \cos\theta
                \end{array})\vecTri{\bm{t}}{\bm{n}}{\bm{b}}$, $\vecTri{\bm{t}}{\bm{n}}{\bm{b}} = \qty(\begin{array}{ccc}
                    1 & 0 & 0 \\
                    0 & \cos\theta & -\sin\theta \\
                    0 & \sin\theta & \cos\theta
                \end{array})^{-1}\vecTri{\bm{a}_1}{\bm{a}_2}{\bm{a}_3}$ and $\qty(\begin{array}{ccc}
                    1 & 0 & 0 \\
                    0 & \cos\theta & -\sin\theta \\
                    0 & \sin\theta & \cos\theta
                \end{array})^{-1}=\matrixTT{1}{0}{0}{0}{\cos\theta}{\sin\theta}{0}{-\sin\theta}{\cos\theta}$, we have
                    $$\begin{array}{rcl}
                        \vecTri{\bm{a}_1'}{\bm{a}_2'}{\bm{a}_3'} &=& \theta'\qty(\begin{array}{ccc}
                            0 & 0 & 0 \\
                            0 & -\sin\theta & -\cos\theta \\
                            0 & \cos\theta & -\sin\theta
                        \end{array})\vecTri{\bm{t}}{\bm{n}}{\bm{b}} + \qty(\begin{array}{ccc}
                            1 & 0 & 0 \\
                            0 & \cos\theta & -\sin\theta \\
                            0 & \sin\theta & \cos\theta
                        \end{array})\qty(\begin{array}{ccc}
                             0 & \kappa & 0  \\
                             -\kappa & 0 & \tau \\
                             0 & -\tau & 0
                        \end{array}) \vecTri{\bm{t}}{\bm{n}}{\bm{b}}\\
                     &=& \matrixTT{0}{\kappa}{0}{-\kappa\cos\theta}{\sin\theta(\tau - \theta')}{\cos\theta(\tau - \theta')}{-\kappa\sin\theta}{-\cos\theta(\tau-\theta')}{\sin\theta(\tau-\theta')}\matrixTT{1}{0}{0}{0}{\cos\theta}{\sin\theta}{0}{-\sin\theta}{\cos\theta}\vecTri{\bm{a}_1}{\bm{a}_2}{\bm{a}_3}\\
                     &=& \matrixTT{0}{\kappa\cos\theta}{\kappa\sin\theta}{-\kappa\cos\theta}{0}{\tau-\theta'}{-\kappa\sin\theta}{-(\tau-\theta')}{0}\vecTri{\bm{a}_1}{\bm{a}_2}{\bm{a}_3}.
                    \end{array}$$
            \end{proof}
    
            We compute $\bm{a}_1,\bm{a}_2,\bm{a}_3$ using Picard iteration. Taylor expansions of $\bm{a}_1,\bm{a}_2,\bm{a}_3$ are obtained to calculate recurrence relations. As a initial condition, a direction of the singular locus at the origin set $\bm{t}=(1,0,0)$, $\bm{n}=(0,1,0)$ and $\bm{b}=(0,0,1)$, then $\bm{a}_1(0),\bm{a}_2(0),\bm{a}_3(0)$ are calculated for the definition.
                $$\vecTri{\bm{a}_1(0)}{\bm{a}_2(0)}{\bm{a}_3(0)}  =  \qty(\begin{array}{ccc}
                    1 & 0 & 0 \\
                    0 & \cos\theta(0) & -\sin\theta(0) \\
                    0 & \sin\theta(0) & \cos\theta(0)
                \end{array})\vecTri{\bm{t}}{\bm{n}}{\bm{b}} = \matrixTT{1}{0}{0}{0}{\cos\theta_0}{-\sin\theta_0}{0}{\sin\theta_0}{\cos\theta_0}$$
            We assume a sequence $\qty{(\bm{a}_1,\bm{a}_2,\bm{a}_3)^t_n}_{n=1}^\infty$ converges to $(\bm{a}_1,\bm{a}_2,\bm{a}_3)^t$ and $(\bm{a}_1,\bm{a}_2,\bm{a}_3)^t_1 = (\bm{a}_1(0),\bm{a}_2(0),\bm{a}_3(0))^t$. Then, the next recurrence relations follow from the Lemma \ref{lem:SeFrType}.
                $$\frac{d}{ds}\vecTri{\bm{a}_1}{\bm{a}_2}{\bm{a}_3}_{n+1} = \matrixTT{0}{\kappa\cos\theta}{\kappa\sin\theta}{-\kappa\cos\theta}{0}{\tau-\theta'}{-\kappa\sin\theta}{-(\tau-\theta')}{0}\vecTri{\bm{a}_1}{\bm{a}_2}{\bm{a}_3}_n$$
            We rewrite the recurrence relations as a integrate form.
                $$\vecTri{\bm{a}_1}{\bm{a}_2}{\bm{a}_3}_{n+1} = \qty(\begin{array}{ccc}
                1 & 0 & 0 \\
                0 & \cos\theta(0) & -\sin\theta(0) \\
                0 & \sin\theta(0) & \cos\theta(0)
                \end{array}) + \int_0^s\matrixTT{0}{\kappa\cos\theta}{\kappa\sin\theta}{-\kappa\cos\theta}{0}{\tau-\theta'}{-\kappa\sin\theta}{-(\tau-\theta')}{0}\vecTri{\bm{a}_1}{\bm{a}_2}{\bm{a}_3}_n ds$$
            Calculating these recurrence relations, we obtain an asymptotic expression of $\bm{a}_2$ and $\bm{a}_3$.
                \begin{equation}\label{eq:a2}
                    \begin{split}
                        j^3\bm{a}_2 &= \vecTri{0}{\cos\theta_0}{-\sin\theta_0} + \vecTri{-\kappa_0\cos\theta_0}{(\tau_0-\theta_1)\sin\theta_0}{(\tau_0-\theta_1)\cos\theta_0}s - \vecTri{\kappa_0(\tau_0-2\theta_1)\sin\theta_0+\kappa_1\cos\theta_0}{-(\tau_1-\theta_2)\sin\theta_0+\qty{(\tau_0-\theta_1)^2+\kappa_0^2}\cos\theta_0}{-(\tau_0-\theta_1)^2\sin\theta_0-(\tau_1-\theta_2)\cos\theta_0}\frac{s^2}{2}\\
                        & + \vecTri{-\qty{\kappa_0(2\tau_1-3\theta_2)} \sin\theta_0 + \qty(\kappa_0\qty{\kappa_0^2+
                        \tau_0^2-3\theta_1(\tau_0-\theta_1)}-\kappa_2) \cos\theta_0}
                        {-\qty{(\tau_0-\theta_1)^3+\kappa_0^2(\tau_0-3\theta_1)-(\tau_2-\theta_3)} \sin\theta_0 - 3\qty{\tau_1(\tau_0-\theta_1)-\theta_2(\tau_0-\theta_1)+\kappa_0\kappa_1} \cos\theta_0}
                        {3(\tau_0-\theta_1)(\tau_1-\theta_2) \sin\theta_0 - \qty{(\tau_0-\theta_1)^3+\kappa_0^2\tau_0-(\tau_2-\theta_3)
                        } \cos\theta_0}\frac{s^3}{6}
                    \end{split}
                \end{equation}
                \begin{equation}\label{eq:a3}
                    j^2\bm{a}_3 = \vecTri{0}{\sin\theta_0}{\cos\theta_0}
                    - \vecTri{\kappa_0\sin\theta_0}{(\tau_0-\theta_1)\cos\theta_0}{-(\tau_0-\theta_1)\sin\theta_0}s
                    - \vecTri{\kappa_1\sin\theta_0-\kappa_0 (\tau_0-2\theta_1)  \cos\theta_0}
                    {\qty{(\tau_0-\theta_1)^2+\kappa_0^2}\sin\theta_0+(\tau_1-\theta_2)\cos\theta_0}
                    {-(\tau_1-\theta_2)\sin\theta_0+(\tau_0-\theta_1)^2\cos\theta_0}\frac{s^2}{2}
                \end{equation}
    
            Similarly, the singular locus $\gamma$ is calculated. We assume a sequence $\qty{(\bm{t},\bm{b},\bm{n})^t_n}_{n=1}^\infty$ converges to $(\bm{t},\bm{n},\bm{b})^t$ and $(\bm{t},\bm{n},\bm{b})^t_1 = (\bm{t}_1(0),\bm{n}_2(0),\bm{b}_3(0))^t$. Then, the next recurrence relations hold due to the Frenet-Serret Theorem.
                $$\frac{d}{ds}\vecTri{\bm{t}}{\bm{n}}{\bm{b}}_{n+1} = \qty(\begin{array}{ccc}
                     0 & \kappa & 0  \\
                     -\kappa & 0 & \tau \\
                     0 & \tau & 0
                \end{array}) \vecTri{\bm{t}}{\bm{n}}{\bm{b}}_n$$
            We rewrite the recurrence relations as a integrate form.
                $$\vecTri{\bm{t}}{\bm{n}}{\bm{b}}_{n+1} = \matrixTT{1}{0}{0}{0}{1}{0}{0}{0}{1} + \int_0^s\qty(\begin{array}{ccc}
                 0 & \kappa & 0  \\
                 -\kappa & 0 & \tau \\
                 0 & \tau & 0
            \end{array})\vecTri{\bm{t}}{\bm{n}}{\bm{b}}_n ds$$
            Calculating these recurrence relations, we obtain an asymptotic expression of $\gamma$.
                \begin{equation*}
                    \begin{split}
                        j^4\bm{t} = \vecTri{1}{0}{0} + &\vecTri{0}{\kappa_0}{0}s + \vecTri{-\kappa_0^2}{\kappa_1}{\kappa_0\tau_0}\frac{s^2}{2} - \vecTri{3\kappa_0\kappa_1}{\kappa_0(\tau_0^2-\kappa_0^2)-\kappa_2}{-(\kappa_0\tau_1+2\kappa_1\tau_0)}\frac{s^3}{6}\\
                        & + \vecTri{\kappa_0^2(\kappa_0^2+\tau_0^2)-4\kappa_0\kappa_2-3\kappa_1^2}{-3\kappa_0(\tau_0\tau_1+2\kappa_0\kappa_1)-3\kappa_1\tau_0^2+\kappa_3}{-\kappa_0(\tau_0^3-\tau_2+\kappa_0^2\tau_0) + 3(\kappa_1\tau_1+\kappa_2\tau_0)}\frac{s^4}{24}
                    \end{split}
                \end{equation*}
            The singular locus $\gamma$ is integrated the tangent vector $\bm{t}$, thus
                \begin{equation}\label{eq:gamma}
                    \begin{split}
                        j^5\gamma = \vecTri{1}{0}{0}s + &\vecTri{0}{\kappa_0}{0}\frac{s^2}{2} + \vecTri{-\kappa_0^2}{\kappa_1}{\kappa_0\tau_0}\frac{s^3}{6} - \vecTri{3\kappa_0\kappa_1}{\kappa_0(\tau_0^2-\kappa_0^2)-\kappa_2}{-(\kappa_0\tau_1+2\kappa_1\tau_0)}\frac{s^4}{24}\\
                        & + \vecTri{\kappa_0^2(\kappa_0^2+\tau_0^2)-4\kappa_0\kappa_2-3\kappa_1^2}{-3\kappa_0(\tau_0\tau_1+2\kappa_0\kappa_1)-3\kappa_1\tau_0^2+\kappa_3}{-\kappa_0(\tau_0^3-\tau_2+\kappa_0^2\tau_0) + 3(\kappa_1\tau_1+\kappa_2\tau_0)}\frac{s^5}{120}
                    \end{split}
                \end{equation}
    

        \subsubsection{The first fundamental form and the second fundamental form }
            \begin{lem}
                We obtain the following expressions of the first fundamental form.
                    $$E = 1-\kappa\cos\theta t^2 - \frac{b_3}{3}\kappa\sin\theta t^3 + O^4(t),\ \ \ F = O^4(t),\ \ \ G = t^2$$
            \end{lem}
    
            \begin{proof}
                Since
                    $$\begin{array}{rcl}
                         f_t &=& \pardiff{}{t}\qty(\gamma(s) + a(s,t)\bm{a}_2+b(s,t)\bm{a}_3) \\
                         &=& a_t\bm{a}_2+b_t\bm{a}_2 \\
                         &=& \D \qty{t + \sum_{k=3}^{m-1} a_{k+1} \frac{t^k}{k!}}\bm{a}_2 + \qty{\sum_{k=2}^{m-1} b_{k+1} \frac{t^k}{k!}}\bm{a}_3 + O^m(t)\\
                         &=& (\bm{a}_1, \bm{a}_2, \bm{a}_3)\vecTri{0}{t+a_4\frac{t^3}{6} + O^4(t)}{b_3\frac{t^2}{2} + b_4\frac{t^3}{6} + O^4(t)}
                    \end{array}$$
                    $$\begin{array}{rcl}
                         f_s &=& \pardiff{}{s}\qty(\gamma(s)+a(s,t)\bm{a}_2(s) +b(s,t)\bm{a}_3(s)) \\
                         &=& \bm{a}_1 + (a_s\bm{a}_2 + a\bm{a}'_2) + (b_s\bm{a}_3 + b\bm{a}_3')\\
                         &=& \bm{a}_1 + a_s\bm{a}_2 + b_s\bm{a}_3 + a(-\kappa\cos\theta\bm{a}_1 + (\tau - \theta')\bm{a}_3) + b(-\kappa\sin\theta \bm{a}_1 - (\tau-\theta')\bm{a}_2)\\
                         &=& (\bm{a}_1,\bm{a}_2,\bm{a}_3)\vecTri{1-\kappa(a\cos\theta - b\sin\theta)}{a_s-b(\tau-\theta')}{b_s + a(\tau-\theta')}\\
                        &=& (\bm{a}_1,\bm{a}_2,\bm{a}_3)\vecTri{1-\kappa\cos\theta\frac{t^2}{2} + O^3(t)}{O^3(t)}{(\tau-\theta')\frac{t^2}{2} + O^3(t)}
                    \end{array}$$
                    $$a_s = \sum_{k=4}^{m}a_k'(s)\frac{t^k}{k!} + O^{m+1}(t),\ \ \ b_s = \sum_{k=3}^{m}b'_k(s)\frac{t^k}{k!} + O^{m+1}(t)$$
                we have
                    $$\begin{array}{rcl}
                         E &=& \qty(1-\kappa(a\cos\theta - b\sin\theta))^2 + \qty(a_s-b(\tau-\theta'))^2 + \qty(b_s + a(\tau-\theta'))^2 \\
                         &=& 1 -\kappa\cos\theta t^2 - \frac{b_3}{3}\kappa\sin\theta t^3 + O^4(t)\\
                         F &=& \D \qty{t + \sum_{k=3}^{m-1} a_{k+1} \frac{t^k}{k!}}\qty{a_s-b(\tau-\theta')} + \qty{\sum_{k=2}^{m-1} b_{k+1} \frac{t^k}{k!}}\qty{b_s + a(\tau-\theta')}\\
                         &=& O^4(t) \\
                         G &=& \langles{f_t,f_t} \ =\ t^2
                    \end{array}$$
            \end{proof}
    
            \begin{lem} \label{lem:EFG}
                We obtain the following expression of a norm $|f_s\times f_t|$ and the unit normal vector $\nu$.
                    $$|f_s\times f_t|^2 = t^2\qty(1-\kappa\cos\theta t^2 + O^3(t))$$
                    $$\nu=(\bm{a}_1, \bm{a}_2, \bm{a}_3)\vecTri{-(\tau-\theta')\frac{t^2}{2} + O^3(t)}{-b_3\frac{t}{2}-b_4\frac{t^2}{6} + O^3(t)}{1 - \frac{b_3^2}{8}t^2 + O^3(t)}$$
            \end{lem}
    
            \begin{proof}
                We compute $EG-F^2$, then
                    $$|f_s\times f_t|^2 = EG-F^2 = t^2(1 -\kappa\cos\theta t^2 + O^3(t)).$$
                Since
                    $$\begin{array}{rcl}
                         \bm{a}_2'\times\bm{a}_2 &=& (-\kappa\cos\theta\bm{a}_1+(\tau-\theta')\bm{a}_3)\times\bm{a}_2\\
                         &=& -\kappa\cos\theta\bm{a}_1\times\bm{a}_2 + (\tau-\theta')\bm{a}_3\times\bm{a}_2\\
                         &=& - (\tau-\theta')\bm{a}_1 - \kappa\cos\theta\bm{a}_3,
                    \end{array}$$
                we have
                    $$\begin{array}{rcl}
                         \D f_s\times \frac{f_t}{t} &=& \D \qty(\bm{a}_1+\bm{a}_2'\frac{t^2}{2} + O^3(t))\times\qty(\bm{a}_2+b_3\bm{a}_3\frac{t}{2}+\qty(a_4\bm{a}_2+b_4\bm{a}_3)\frac{t^2}{6} + O^3(t)) \\
                         &=& \D \bm{a}_1\times\bm{a}_2 + b_3\bm{a}_1\times\bm{a}_3\frac{t}{2} + (3\bm{a}_2'\times\bm{a}_2+a_4\bm{a}_1\times\bm{a}_2+b_4\bm{a}_1\times\bm{a}_3)\frac{t^2}{6} + O^3(t) \\
                         &=& \D \bm{a}_3 - b_3 \bm{a}_2\frac{t}{2}+(3\bm{a}_2'\times\bm{a}_2+a_4\bm{a}_3-b_4\bm{a}_2)\frac{t^2}{6} +  O^3(t) \\
                         &=& \D \bm{a}_3 - b_3 \bm{a}_2\frac{t}{2}+(- 3(\tau-\theta')\bm{a}_1 - b_4\bm{a}_2 + (a_4- 3\kappa\cos\theta) \bm{a}_3)\frac{t^2}{6} +  O^3(t) \\
                         &=& (\bm{a}_1, \bm{a}_2, \bm{a}_3)\vecTri{-(\tau-\theta')\frac{t^2}{2} + O^3(t)}{-b_3\frac{t}{2}-b_4\frac{t^2}{6} + O^3(t)}{1  + (a_4- 3\kappa\cos\theta)\frac{t^2}{6} + O^3(t)}
                    \end{array}$$
                A vector $f_s\times f_t/t$ is divided by a norm of $f_s\times f_t/t$ to use Taylor expansion of $\frac{1}{\sqrt{1+x}}$.
                    $$\begin{array}{rcl}
                         \D \frac{f_s\times f_t/t}{|f_s\times f_t/3t|} &=& \D \frac{f_s\times f_t/t}{\sqrt{1 -\kappa\cos\theta t^2 + O^3(t)}} \\
                         &=&\D \qty(1+\frac{\kappa\cos\theta}{2}t^2+O^3(t))\times(\bm{a}_1, \bm{a}_2, \bm{a}_3)\vecTri{-(\tau-\theta')\frac{t^2}{2}}{-b_3\frac{t}{2}-b_4\frac{t^2}{6}}{1 + (a_4- 3\kappa\cos\theta)\frac{t^2}{6}}\\
                         &=& (\bm{a}_1, \bm{a}_2, \bm{a}_3)\vecTri{-(\tau-\theta')\frac{t^2}{2} + O^3(t)}{-b_3\frac{t}{2}-b_4\frac{t^2}{6} + O^3(t)}{1 - \frac{b_3^2}{8}t^2 + O^3(t)}
                    \end{array}$$   
            \end{proof}
            \begin{lem}
                We obtain the following expressions of the second fundamental form.
                    $$L = \kappa\sin\theta -\frac{b_3}{2}\kappa\cos\theta t + O^2(t),\ \ \ M = (\tau-\theta')t + b_3'\frac{t^2}{2} + O^3(t),\ \ \ N = b_3\frac{t}{2} + b_4\frac{t^2}{3} + O^3(t)$$
            \end{lem}
    
            \begin{proof}
                Since
                    $$\begin{array}{rcl}
                        f_{ss} &=& (\bm{a}_1,\bm{a}_2,\bm{a}_3)'\vecTri{1 + O^2(t)}{O^2(t)}{O^2(t)} + (\bm{a}_1,\bm{a}_2,\bm{a}_3)\vecTri{ O^2(t)}{O^2(t)}{O^2(t)}\\
                        &=& (\bm{a}_1,\bm{a}_2,\bm{a}_3)\qty{\matrixTT{0}{-\kappa\cos\theta}{-\kappa\sin\theta}{\kappa\cos\theta}{0}{-(\tau-\theta')}{\kappa\sin\theta}{\tau-\theta'}{0} \vecTri{1 + O^2(t)}{O^2(t)}{O^2(t)} + \vecTri{ O^2(t)}{O^2(t)}{O^2(t)}}\\
                        &=& (\bm{a}_1,\bm{a}_2,\bm{a}_3)\vecTri{O^2(t)}{\kappa\cos\theta + O^2(t)}{\kappa\sin\theta + O^2(t)}\\
                    \end{array}$$
                    $$\begin{array}{rcl}
                         f_{st} &=& (\bm{a}_1, \bm{a}_2, \bm{a}_3)'\vecTri{0}{t+O^3(t)}{b_3\frac{t^2}{2} + O^3(t)} + (\bm{a}_1, \bm{a}_2, \bm{a}_3)\vecTri{0}{O^3(t)}{b_3'\frac{t^2}{2} + O^3(t)}\\
                         &=& (\bm{a}_1,\bm{a}_2,\bm{a}_3)\qty{\matrixTT{0}{-\kappa\cos\theta}{-\kappa\sin\theta}{\kappa\cos\theta}{0}{-(\tau-\theta')}{\kappa\sin\theta}{\tau-\theta'}{0}\vecTri{0}{t+O^3(t)}{b_3\frac{t^2}{2} + O^3(t)} + \vecTri{0}{O^3(t)}{b_3'\frac{t^2}{2} + O^3(t)}} \\
                         &=& (\bm{a}_1,\bm{a}_2,\bm{a}_3)\vecTri{-\kappa\cos\theta t -b_3\kappa\sin\theta \frac{t^2}{2}+ O^3(t)}{-b_3(\tau-\theta')\frac{t^2}{2} + O^3(t)}{(\tau-\theta')t + b_3'\frac{t^2}{2} + O^3(t)}
                    \end{array}$$
                    $$\begin{array}{rcl}
                         f_{tt} &=& (\bm{a}_1, \bm{a}_2, \bm{a}_3)\vecTri{0}{1+a_4\frac{t^2}{2} + O^3(t)}{b_3t + b_4\frac{t^2}{2} + O^3(t)} 
                    \end{array}$$
                    for Lemma. \ref{lem:SeFrType}, we have
                        $$\begin{array}{rcl}
                             L &=& \langles{f_{ss},\nu}\ =\ \langles{\vecTri{O^2(t)}{\kappa\cos\theta + O^2(t)}{\kappa\sin\theta + O^2(t)}, \vecTri{O^2(t)}{-b_3\frac{t}{2} + O^2(t)}{1 + O^2(t)}} \\
                             &=& \D \kappa\sin\theta - b_3\kappa\cos\theta\frac{t}{2} + O^2(t) 
                        \end{array}$$
                        $$\begin{array}{rcl}
                             M &=& \langles{f_{ss},\nu}\ =\ \langles{\vecTri{-\kappa\cos\theta t -b_3\kappa\sin\theta \frac{t^2}{2}+ O^3(t)}{-b_3(\tau-\theta')\frac{t^2}{2} + O^3(t)}{(\tau-\theta')t + b_3'\frac{t^2}{2} + O^3(t)}, \vecTri{-(\tau-\theta')\frac{t^2}{2} + O^3(t)}{-b_3\frac{t}{2}-b_4\frac{t^2}{6} + O^3(t)}{1 - \frac{b_3^2}{8}t^2 + O^3(t)}} \\
                             &=& \D (\tau-\theta')t + b_3'\frac{t^2}{2} + O^3(t)
                        \end{array}$$
                        $$\begin{array}{rcl}
                            N &=& \langles{f_{ss},\nu}\ =\ \langles{\vecTri{0}{1+a_4\frac{t^2}{2} + O^3(t)}{b_3t + b_4\frac{t^2}{2} + O^3(t)}, \vecTri{-(\tau-\theta')\frac{t^2}{2} + O^3(t)}{-b_3\frac{t}{2}-b_4\frac{t^2}{6} + O^3(t)}{1 - \frac{b_3^2}{8}t^2 + O^3(t)}} \\
                             &=& \D b_3\frac{t}{2} + b_4\frac{t^2}{3} + O^3(t).
                        \end{array}$$
            \end{proof}
    
            \begin{lem}\label{lem:Weingarten}
                We obtain the following expression of the Weingarten matrix.
                    $$W = \matrixDD{\kappa\sin\theta - \frac{b_3}{2}\kappa\cos\theta t + O^2(t)}{(\tau -\theta')t + O^2(t)}{\frac{1}{t}((\tau-\theta') + b_3'\frac{t}{2} + O^2(t))}{\frac{1}{t}(\frac{b_3}{2} + b_4\frac{t}{3} + O^2(t))}$$
            \end{lem}
    
            \begin{proof}
                Weingarten matrix holds next formula;
                    $$(\bm{\nu}_s,\bm{\nu}_t)=-(f_s,f_t)W.$$
                Thus, $W^t=\matrixDD{E}{F}{F}{G}^{-1}\matrixDD{L}{M}{M}{N}$ holds. Then,
                    $$\matrixDD{E}{F}{F}{G}^{-1}\matrixDD{L}{M}{M}{N} \ =\ \frac{1}{EG-F^2}\matrixDD{E}{-F}{-F}{G} \matrixDD{L}{M}{M}{N} \\$$
                    $$\begin{array}{rcl}
                        &=& \D \frac{1}{t^2(1 + O^2(t))}\matrixDD{t^2}{O^4(t)}{O^4(t)}{1 + O^2(t)}\matrixDD{\kappa\sin\theta -\frac{b_3}{2}\kappa\cos\theta t + O^2(t)}{(\tau-\theta')t + b_3'\frac{t^2}{2} + O^3(t)}{(\tau-\theta')t + b_3'\frac{t^2}{2} + O^3(t)}{b_3\frac{t}{2} + b_4\frac{t^2}{3} + O^3(t)} \\
                        &=& \D \frac{1}{t^2(1 + O^2(t))}\matrixDD{\kappa\sin\theta t^2 -\frac{b_3}{2}\kappa\cos\theta t^3 + O^4(t)}{(\tau -\theta')t^3 + O^4(t)}{(\tau-\theta')t + b_3'\frac{t^2}{2} + O^3(t)}{b_3\frac{t}{2} + b_4\frac{t^2}{3} + O^3(t)}\\
                        &=& \matrixDD{\kappa\sin\theta - \frac{b_3}{2}\kappa\cos\theta t + O^2(t)}{(\tau -\theta')t + O^2(t)}{\frac{1}{t}((\tau-\theta') + b_3'\frac{t}{2} + O^2(t))}{\frac{1}{t}(\frac{b_3}{2} + b_4\frac{t}{3} + O^2(t))}
                    \end{array}$$
            \end{proof}
            \begin{cor}\label{cor:GH}
                We obtain the Gaussian curvature $K$ and the mean curvature $H$.
                    $$\begin{array}{rcl}
                         H&=& \frac{1}{t} \qty[ \frac{b_4}{4} + \frac{1}{2}\qty(\frac{b_4}{3} + \kappa\sin\theta)t + O(t^2) ],\\
                        K&=& \frac{1}{t} \qty[ \frac{b_3\kappa\sin\theta}{2} + \qty{ \frac{b_4\kappa\sin\theta}{3} - \qty(\tau-\theta')^2 - \frac{b_3^2\kappa\cos\theta}{4}} t + O(t^2) ].  
                    \end{array}$$
                    $$$$
            \end{cor}
            \begin{proof}
                Since $K = \det W$ and $H = \frac12\tr W$, we obtain the state.
            \end{proof}
        
        \subsubsection{Invariants}
            \begin{lem} \label{lem:invariants}
                Several geometric invariants for cuspidal edges were already defined in \cite{MS}, \cite{MSUY}, \cite{SUY}. Here is a list of these invariants.
                \begin{itemize}
                    \item \textbf{normal curvature} in \cite{SUY} $\kappa_{\bm{\nu}} = \kappa\cos\theta$
                    \item \textbf{singular curvature} in \cite{SUY} $\kappa_s = \kappa\sin\theta$
                    \item \textbf{cuspidal curvature} in \cite{MSUY} $\kappa_c = b_3$
                    \item \textbf{cusp-directional torsion} in \cite{MS} $\kappa_t = \tau-\theta'$
                    \item \textbf{edge inflectional curvature} in \cite{MS} $\kappa_i =\kappa\tau\cos\theta+\kappa'\sin\theta$
                \end{itemize}
            \end{lem}
            \begin{proof}
                Since
                    $$\begin{array}{rcl}
                         f_{ttt} &=& (\bm{a}_1, \bm{a}_2, \bm{a}_3)\vecTri{0}{O^1(t)}{b_3 + O^1(t)} \\
                    \end{array}$$
                    $$\begin{array}{rcl}
                         f_{stt} &=& (\bm{a}_1, \bm{a}_2, \bm{a}_3)\qty{\matrixTT{0}{-\kappa\cos\theta}{-\kappa\sin\theta}{\kappa\cos\theta}{0}{-(\tau-\theta')}{\kappa\sin\theta}{\tau-\theta'}{0}\vecTri{0}{1 + O^1(t)}{O^1(t)} + \vecTri{0}{ O^1(t)}{O^1(t)}}\\
                         &=& (\bm{a}_1, \bm{a}_2, \bm{a}_3)\vecTri{\kappa\cos\theta+O^1(t)}{O^1(t)}{\tau-\theta'+O^1(t)}
                    \end{array}$$
                    $$\begin{array}{l}
                         f_{sss}\  =\\ 
                         (\bm{a}_1, \bm{a}_2, \bm{a}_3)\qty{\matrixTT{0}{-\kappa\cos\theta}{-\kappa\sin\theta}{\kappa\cos\theta}{0}{-(\tau-\theta')}{\kappa\sin\theta}{\tau-\theta'}{0}\vecTri{O^1(t)}{\kappa\cos\theta + O^1(t)}{\kappa\sin\theta + O^1(t)} + \vecTri{O^1(t)}{\kappa'\cos\theta - \kappa\theta'\sin\theta + O^1(t)}{\kappa'\sin\theta + \kappa\theta'\cos\theta + O^1(t)}}\\
                         =\ (\bm{a}_1, \bm{a}_2, \bm{a}_3)\vecTri{-\kappa^2(\cos^2\theta+\sin^2\theta + O^1(t))}{-\kappa(\tau-\theta')\sin\theta + \kappa'\cos\theta - \kappa\theta'\sin\theta + O^1(t)}{\kappa(\tau-\theta')\cos\theta + \kappa'\sin\theta + \kappa\theta'\cos\theta + O^1(t)}\\
                         =\ (\bm{a}_1, \bm{a}_2, \bm{a}_3)\vecTri{-\kappa^2+O^1(t)}{-\kappa\tau\sin\theta+\kappa'\cos\theta+O^1(t)}{\kappa\tau\cos\theta+\kappa'\sin\theta+O^1(t)},
                    \end{array}$$
                we compute according to the definition.
                    $$\begin{array}{c@{\hskip3pt}rcccl}
                         \bullet & \kappa_{\bm{\nu}} &=& \det(f_s,\ f_{ss},\ \bm{a}_3)|_{t=0} &=& \left.\qty|\begin{array}{ccc}
                              1 & 0 & 0 \\
                              0 & \kappa\cos\theta & 0 \\
                              0 & \kappa\sin\theta & 1 \\
                         \end{array}| + O^1(t)\right|_{t=0}\\
                         \bullet & \kappa_s &=& \langles{f_{ss},\bm{\nu}}|_{t=0} &=& \left.\langles{\kappa\cos\theta\bm{a}_2+\kappa\sin\theta\bm{a}_3,\bm{a}_3} + O^1(t)\right|_{t=0}\\
                         \bullet & \kappa_c &=& \det(f_s,\ f_{tt},\ f_{ttt})|_{t=0} &=& \left.\qty|\begin{array}{ccc}
                              1 & 0 & 0 \\
                              0 & 1 & 0 \\
                              0 & 0 & b_3 \\
                         \end{array}| + O^1(t)\right|_{t=0}\\
                         \bullet & \kappa_t &=& \det(f_s,\ f_{tt},\ f_{stt})|_{t=0} &=& \left.\qty|\begin{array}{ccc}
                              1 & 0 & \kappa\cos\theta \\
                              0 & 1 & 0 \\
                              0 & 0 & \tau-\theta' \\
                         \end{array}| + O^1(t)\right|_{t=0} \\
                         \bullet & \kappa_i &=& \det(f_s,\ f_{tt},\ f_{sss})|_{t=0} &=& \left.\qty|\begin{array}{ccc}
                              1 & 0 & -\kappa^2 \\
                              0 & 1 & -\kappa\tau\sin\theta+\kappa'\cos\theta \\
                              0 & 0 & \kappa\tau\cos\theta+\kappa'\sin\theta \\
                         \end{array}| + O^1(t)\right|_{t=0} \\
                    \end{array}$$
            \end{proof}
    
            The coefficient $b_3$, which determinant cupidal $S_k$ is geometric invariant $\kappa_c$.
            
            \begin{proof}[Proof of Proposition \ref{prop:DegCSk}.]
                According to \cite{SSk}, a map $f:\R^2\rightarrow\R^3$ is diffeomorphic to cuspidal $S_1$ singularity at $\bm{0}$ if and only if:
                \begin{itemize}
                    \item[(a)]  The origin $\bm{0}$ is a non-degenerate singular point and the null vector is transverse to $S(f)$ at $\bm{0}$,
                    \item[(b)] There exists a curve $c:(-\varepsilon,\varepsilon),\bm{0} \rightarrow \R^2,\bm{0}$ such that $c'(0)$ is parallel to $\eta(\bm{0})$, $\hat{c}'(0)=0$, $\hat{c}''(0)\not=0$ and there exists $l\in\R$ satisfying $\hat{c}'''(0) = l\hat{c}''(0)$ and the determinant $\det \qty(\gamma'\ \hat{c}'',\ 3\hat{c}^{(5)}-10lc^{(4)})(0)\not=0$, where $\hat{c}=f\circ c$, and
                    \item[(c)] The function $\psi(0)=\psi'(0)=0$ and $\psi''(0)\not=0$, where $\psi$ is the function defined by $\det \qty(\frac{d}{dt}\gamma(t), \nu|_{ S(f)}, d\nu_{S(f)}(\eta(t)))$.
                \end{itemize}
                Note that $S(f)=\{t=0\}$, $\eta =\partial_t$, so (a) holds automatically in Fukui's normal form. First, we consider (c). 
        
                    $$\gamma' =\bm{t}=\bm{a}_1$$
                    $$\nu|_{S(f)} = (O^1(t))\bm{a}_1 + (O^1(t))\bm{a}_2 + (1 + O^1(t))\bm{a}_3|_{S(f)} = \bm{a}_3$$
                    \begin{equation*}
                        \begin{split}
                            d&\nu_{S(f)}(\eta(t)) = (\nu_s,\ \nu_t)\left.\vecDual{0}{1}\right|_{S(f)} = (f_s,\ f_t)W\left.\vecDual{0}{1}\right|_{S(f)}\\
                            &= (\bm{a}_1,\bm{a}_2,\bm{a}_3)\qty(\begin{array}{cc}
                                 O^0(t) & O^2(s,t) \\
                                 O^0(t) & t + O^2(s,t) \\
                                 O^0(t) & O^2(s,t)
                            \end{array})\qty(\begin{array}{cc}
                                 * & O^1(t) \\
                                 * & \frac{1}{t}\qty(g(s) +O^1(t))
                            \end{array}).\left.\vecDual{0}{1}\right|_{S(f)}\\
                            &=(\bm{a}_1,\bm{a}_2,\bm{a}_3)\left.\qty(\begin{array}{cc}
                                 * & O^1(t) \\
                                 * & -g(s)+O^1(t)\\
                                 * & O^1(t)
                            \end{array})\vecDual{0}{1}\right|_{S(f)} =-g(s)\bm{a}_2,
                        \end{split}
                    \end{equation*}
                where $g(s)=\frac{b_{03}}{2}+\frac{b_{13}}{2}s+\frac{b_{23}}{4}s^2+O^3(s)$. Thus, $\psi(s)=g(s)$, $\psi(0)=b_{03}=b_3(0)$, $\psi'(0)=b_{13}=b_3'(0)$ and $\psi^{(2)}(s)=b_{23}=b_3^{(2)}$. That is $b_3'(0)=b_3''(0)=0$ and $b_3^{(2)}\not=0$. Based on this result, we may assume $b_{03}=b_{13}=0$.
            
                We show that the curve $c(x)=(\frac{c_2}{2}x^2+O^6(x),x)$, $c_2\not=0$, whose derivative at 0 $\frac{d}{dx}c(0)$ is parallel to $\partial_t$, satisfies condition (c). We substitute $c(x)$ into $f(s,t)$ term by term, using $f(s,t) = \gamma(s)+a(s,t)\bm{a}_2+b(s,t)\bm{a}_3$.
                    $$\gamma(s)=(\bm{e}_1,\bm{e}_2,\bm{e}_3)\vecTri{s + O^6_{\bm{w}}(s)}
                    {\frac{\kappa_0}{2}s^2 + O^6_{\bm{w}}(s)}
                    {O^6_{\bm{w}}(s)}$$
                    
                    $$a(s,t)\bm{a}_2(s)=(\bm{e}_1,\bm{e}_2,\bm{e}_3)\vecTri{-\frac{1}{2}\kappa_0\cos\theta_0st^2 + O^6_{\bm{w}}(s,t)}
                    {\frac{\cos\theta_0}{2}t^2 + \frac{\sin\theta_0}{2}(\tau_0-\theta_1)st^2 + O^6_{\bm{w}}(s,t)}
                    {-\frac{\sin\theta_0}{2}t^2 + \frac{\cos\theta_0}{2}(\tau_0-\theta_1)st^2 + O^6_{\bm{w}}(s,t)}$$
                    
                    $$b(s,t)\bm{a}_3(s)=(\bm{e}_1,\bm{e}_2,\bm{e}_3)\vecTri{O^6_{\bm{w}}(s,t)}
                    {\frac{b_{04}}{24}\sin\theta_0t^4 + \frac{b_{05}}{120}\sin\theta_0t^5 + O^6_{\bm{w}}(s,t)}
                    {\frac{b_{04}}{24}\cos\theta_0t^4 + \frac{b_{05}}{120}\cos\theta_0t^5 + O^6_{\bm{w}}(s,t)}$$
                We express these Taylor expansions in a weighted form with respect to $(s,t)$, assigning weights $\bm{w}=(2,1)$. The following are $\gamma$, $a\bm{a}_2$ and $b\bm{a}_3$ after substituting $c(x)$.
                    $$\gamma\circ c(x) = (\bm{e}_1,\bm{e}_2,\bm{e}_3)\vecTri{\frac{c_2}{2}x^2+O^6(x)}{\frac{c_2^2}{8}\kappa_0x^4 + O^6(x)}{O^6(x)}$$
                    $$a\bm{a}_2\circ c(x) = (\bm{e}_1,\bm{e}_2,\bm{e}_3)\vecTri{-\frac{c_2}{4}\kappa_0\cos\theta_0x^4 + O^6(x)}
                    {\frac{1}{2}\cos\theta_0x^2+\frac{c_2}{2}(\tau_0-\theta_1) \sin\theta_0 x^4 + O^5(x)}
                    {-\frac{1}{2}\sin\theta_0x^2+\frac{c_2}{4}(\tau_0-\theta_1)\cos\theta_0x^4 + O^6(x)}$$
                    $$b\bm{a}_3\circ c(x) = (\bm{e}_1,\bm{e}_2,\bm{e}_3)\vecTri{O^6(x)}{\frac{b_{04}}{24}\sin\theta_0x^4+\frac{b_{05}}{120}\sin\theta_0x^5 + O^6(x)}{\frac{b_{04}}{24}\cos\theta_0x^4 + \frac{b_{05}}{120}\cos\theta_0x^5 + O^6(x)}$$
                Thus, $\hat{c}'(0) = \left.\frac{d}{dx}(\gamma+a\bm{a}_2+b\bm{a}_3)\circ c\right|_{x=0}=\bm{0}$, $\hat{c}''(0)=c_2\bm{e}_1\cos\theta_0\bm{e}_2-\sin\theta_0\bm{e}_3$ and $\hat{c}'''(0)=\bm{0}$, so $l=0$ since $\hat{c}'''(0)=l\hat{c}''(0)$. We have $\gamma'(0)=\bm{t}(0)=\bm{e}_1$ from initial condition and 
                    $$3\hat{c}^{(5)}-10l\hat{c}^{(4)}=3\hat{c}^{(5)}=(\bm{e}_1,\bm{e}_2,\bm{e}_3)\vecTri{0}{b_{05}\sin\theta_0}{b_{05}\cos\theta_0}.$$
                Therefore, we obtain $\det \qty(\gamma'\ \hat{c}'',\ 3\hat{c}^{(5)}-10lc^{(4)})(0)=b_{05}$. That is, there exists a curve if $b_{05}\not=0$. Conversely, this calculation shows that there is no such curve if $b_{05}=0$.
            \end{proof}

\section{Via distance square function} \label{Sec:3}
    \subsection{Criterion for $A_k$, $D_k$, $E_k$-singularities}
        We consider a function $g:\R^2,\bm{0} \rightarrow \R, \bm{0}$ whose Taylor expansion at $(0,0)$ is
            $$g(u,v) = \frac{c_{20}}{2} u^2 + \frac{c_{02}}{2}v^2 + \sum_{i,j \geq 3} \frac{c_{ij}}{i!j!} u^iv^j,$$
        and denote homogeneous $\sum_{i+j = k} \frac{c_{ij}}{i!j!} u^i v^j$ of degree $k$.
        
        It is well known that $g$ has an $A_1$-singularity at \Origin{} if and only if the Hessian matrix of $g$ is of full rank, i.e., $c_{20}\not=0$ and $c_{02}\not=0$. If either $c_{20}$ or $c_{02}$ is 0, $g$ has a corank 1 singularity at $(0,0)$, and if both are 0, $g$ has a corank 2 singularity at \Origin{}.
    
        \begin{prop} \label{prop:ACriteria}
            We assume $g(u,v)$ has a corank 1 singularity, so we set $c_{20} = 0$, $c_{02} \not= 0$. Then\\
            (1) The function $g(u,v)$ has an $A_2$ singularity at \Origin{} if and only if $c_{30} \not= 0$.\\
            (2) The function $g(u,v)$ has an $A_3$ singularity at \Origin{} if and only if $c_{30} = 0$ and $c_{02}c_{40} - 3c_{21}^2 \not= 0$.\\
            (3) The function $g(u,v)$ has an $A_4$ singularity at \Origin{} if and only if $c_{30} = 0$, $c_{02}c_{40} - 3c_{21}^2 = 0$ and
                $$c_{02}^2c_{50} - 10 c_{02} c_{21} c_{31} + 15c_{12}c_{21}^2 \not=0.$$
            (4) The function $g(u,v)$ has an $A_5$ singularity at \Origin{} if and only if $c_{30} = 0$, $c_{02}c_{40} - 3c_{21}^2 = 0$, 
                $$c_{02}^2c_{50} - 10 c_{02} c_{21} c_{31} + 15c_{12}c_{21}^2 =0 \text{ and }$$
                $$c_{02}^2 (3c_{21}c_{41} + 2c_{31}^2) - 3c_{02}c_{21}(4c_{12}c_{31} + 3c_{21} c_{22}) + 3c_{21}^2(c_{03}c_{21} + 6c_{12}^2) \not = 0$$
        \end{prop}
        \begin{proof}
            The proof is given by explicit coordinate changes $\varphi=(id_\R,\varphi^2):\R^2,0\rightarrow\R^2,0$ satisfying
                \begin{equation*}
                    \begin{split}
                        j^5 \varphi^2(u,v) = & v - \qty( \frac{c_{21}}{c_{02}} \frac{u^2}{2!} +\frac{c_{12}}{2c_{02}} uv + \frac{c_{03}}{3c_{02}} \frac{v^2}{2!} )\\
                         & - \qty( \frac{c_{02}c_{31} - 3c_{12}c_{21}}{c_{02}^2} \frac{u^3}{3!} + \frac{c_{02}c_{22} - 2c_{03}c_{21} - 3c_{12}^2}{4c_{02}^2} \frac{u^2v}{2!1!} + \frac{c_{02}c_{13} - 2c_{03}c_{12}}{3c_{02}^2} \frac{uv^2}{1!2!} + \frac{3c_{02}c_{04} - 5c_{03}^2}{12c_{02}^2} \frac{v^3}{3!})\\
                         & - \sum_{i+j = 4} \varphi^2_{ij}\frac{u^iv^j}{i!j!} + \sum_{i+j = 5} \varphi^2_{ij}\frac{u^iv^j}{i!j!}.
                    \end{split}
                \end{equation*}
            Here $\varphi^2_{ij}$ are given explicitly for $i+j=4,5$, as listed below.
                $$\begin{array}{ccl}
                    \varphi^2_{40} &= &  {1\over{a_{02}^ 3}} \qty{{{a_{02}^2a_{41} - a_{02}\left(4a_{12}a_{31}+6a_{21} a_{22}\right) + 3a_{21}(a_{03}a_{21} + 4a_{12}^2)}}}\\
                    
                    \varphi^2_{31} &= &  {1\over{8a_{02}^3}} \qty{{4a_{02}^2a_{32} -2 a_{02}\left(2a_{03}a_{31}+9 a_{12}a_{22}+6a_{13}a_{21}\right) + 15a_{12}(2a_{03}a_{21}+ a_{12}^2)}}\\
                    
                    \varphi^2_{22} &= &  {1\over{3a_{02}^3}} \qty{{a_{02}^2a_{23} - a_{02}\left(2a_{03}a_{22} + a_{04} a_{21} + 4a_{12}a_{13}\right) + 2a_{03}(a_{03}a_{21} + 3a_{12}^ 2)}}\\
                    
                    \varphi^2_{13} &= &  {1\over{24a_{02}^3}} \qty{{6a_{02}^2a_{14} - 5a_{02}\left(4a_{03}a_{13} + 3a_{04}a_{12}\right) + 35a_{03}^2a_{12}}}\\
                    
                    \varphi^2_{04} &= &  {1\over{45a_{02}^3}} \qty{{9a_{02}^2a_{05}-45a_{02}a_{03}a_{04}+40a_{03}^3 }}\\
                    
                \end{array}$$
                $$\begin{array}{ccl}
                    \varphi^2_{50} &= &  {1\over{a_{02}^4}} \{a_{02}^2\left(5a_{12}a_{41}+10a_{21}a_{32}+10 a_{22}a_{31}\right) \\
                    & & + a_{02}\left(\left(-10a_{03}a_{21}-20 a_{12}^2\right)a_{31}-60a_{12}a_{21}a_{22}-15a_{13} a_{21}^2\right)+45a_{03}a_{12}a_{21}^2+60a_{12}^3a_{21} \}\\
                    
                    \varphi^2_{41} &= &  {1\over{16a_{02}^4}} \{a_{02}^2\left(8a_{03}a_{41}+48a_{12}a_{32}+32 a_{13}a_{31}+48a_{21}a_{23}+36a_{22}^2\right)\\
                    & & +a_{02}\left(- 80a_{03}a_{12}a_{31}+\left(-120a_{03}a_{21}-180a_{12}^2 \right)a_{22}-24a_{04}a_{21}^2-240a_{12}a_{13}a_{21} \right)\\
                    & & +60a_{03}^2a_{21}^2+420a_{03}a_{12}^2a_{21}+105 a_{12}^4\}\\
                    
                    \varphi^2_{32} &= &  {1\over{3a_{02}^4}} \{a_{02}^2\left(2a_{03}a_{32}+a_{04}a_{31}+6a_{12} a_{23}+6a_{13}a_{22}+3a_{14}a_{21}\right) \\
                    & & +a_{02}\left(-2 a_{03}^2a_{31}-18a_{03}a_{12}a_{22}+\left(-12a_{03}a_{13}-9 a_{04}a_{12}\right)a_{21}-18a_{12}^2a_{13}\right)+24a_{03} ^2a_{12}a_{21}+24a_{03}a_{12}^3\}\\
                    
                    \varphi^2_{23} &= &  {1\over{48a_{02}^4}} \{a_{02}^2\left(40a_{03}a_{23}+30a_{04}a_{22}+12 a_{05}a_{21}+60a_{12}a_{14}+40a_{13}^2\right)\\
                    & & +a_{02}\left(- 70a_{03}^2a_{22}-70a_{03}a_{4}a_{21}-280a_{03}a_{12} a_{13}-105a_{04}a_{12}^2\right)+70a_{03}^3a_{21}+315a_{3}^2 a_{12}^2\}\\
                    
                    \varphi^2_{14} &= &  {1\over{45a_{02}^4}} \{a_{02}^2\left(45a_{03}a_{14}+45a_{04}a_{13}+27 a_{05}a_{12}\right)+a_{02}\left(-120a_{03}^2a_{13}-180a_{03} a_{04}a_{12}\right)+200a_{03}^3a_{12}\}\\
                    
                    \varphi^2_{05} &= &  {1\over{144a_{02}^4}} \{a_{02}^2\left(168a_{03}a_{05}+105a_{04}^2\right)-630 a_{02}a_{03}^2a_{04}+385a_{03}^4\}\\
                \end{array}$$
            Then, the composition of the function $g$ with the coordinate change $\varphi$ is given as follows:
                \begin{equation*}
                    \begin{split}
                        j^6 g\circ \varphi (u,v) = & \ c_{20} \frac{u^2}{2} + c_{02}\frac{v^2}{2} + c_{30}\frac{u^3}{3!} + \frac{c_{02}c_{40} - 3c_{21}^2}{24c_{02}} u^4 + \frac{c_{02}^2c_{50} - 10 c_{02} c_{21} c_{31} + 15c_{12}c_{21}^2}{120c_{02}^2} u^5\\
                        & + \frac{c_{02}^2 (3c_{21}c_{41} + 2c_{31}^2) - 3c_{02}c_{21}(4c_{12}c_{31} + 3c_{21} c_{22}) + 3c_{21}^2(c_{03}c_{21} + 6c_{12}^2)}{144c_{02}^3} u^6
                    \end{split}
                \end{equation*}
        \end{proof}
    
        \begin{prop}\label{prop:DCriterion}
            We assume $g(u,v)$ has a corank 2 singularity, so we set $c_{20} = c_{02} = 0$. Then $g(u,v)$ has a $D_4$ singularity at \Origin{} if and only if
                $$\qty|\begin{array}{cccc}
                     c_{30} & 2c_{21} & c_{12} & 0  \\
                     0 & c_{30} & 2c_{21} & c_{12} \\
                     c_{21} & 2c_{12} & c_{03} & 0 \\
                     0 & c_{21} & 2c_{12} & c_{03} \\
                \end{array}| \not = 0.$$
        \end{prop}
    
        \begin{proof}
            A function $g$ is $\mathcal{R}$-equivalent to $(x^2 \pm y^2)y$, that is, $g$ has a $D_4$ singularity, if and only if a cubic form $H_3(u,v)$ of $g$ has only simple roots. Thus, the discriminant of $H_3(u,v)$ is not zero.
        \end{proof}
    
        \begin{prop}\label{prop:ECriteria}
            We assume $g(u,v)$ has a corank 2 singularity, so let $c_{20} = c_{02} = 0$ and the coefficients $c_{21}$, $c_{12}$ and $c_{03}$ are also 0. Then \\
            (1) The function $g(u,v)$ has a $E_6$ singularity at \Origin{} if and only if $c_{30} \not= 0$ and $c_{04}\not=0$.\\
            (2) The function $g(u,v)$ has a $E_7$ singularity at \Origin{} if and only if $c_{30} \not= 0$, $c_{04} =0$ and $c_{14}\not=0$.\\
            (3) The function $g(u,v)$ has a $E_8$ singularity at \Origin{} if and only if $c_{30} \not= 0$, $c_{04} = c_{13} = 0$, $c_{05} \not= 0$.
        \end{prop}
    
        \begin{proof}
            The cubic form $H_3(u,v)$ of $g$ has a triple root $(u,v) = (0, s)$ for all $s\in \R$ since $g(u,v) = \frac{c_{30}}{6} u^3 + O^4(u,v)$. Then $g$ is $\mathcal{R}$-equivalent to $E_6$-singularity if and only if a triple root is not a root of the quadric form $H_4(u,v)$ of $g(u,v)$.  The triple root is a simple root of $H_4(u,v)$ if and only if it is an $E_7$-singularity, and the triple root is a double root of $H_4(u,v)$ and not a root of $H_5(u,v)$ if and only if $E_8$-singularity. 
        \end{proof}
    
    \subsection{Singularities of the distance squared unfolding}
        Let $f:\R^2,\bm{0}\rightarrow \R^3,\bm{0}$ be parameterized in Fukui's normal form.
        To study singularities of parallel surfaces of $f$, we consider the \textbf{distance squared unfolding} $\Phi:\R^2\times\R^3\times\R \rightarrow \R$ defined by $(s,t,x,y,z, \varepsilon) \mapsto || (x,y,z) - f(s,t) || - \varepsilon^2$.
    
        The discriminant set $\mathcal{D}(\varphi)$ is given by
            $${\small \begin{array}{@{\hskip1pt}cc@{\hskip3pt}c@{\hskip3pt}l}
                & \mathcal{D}(\Phi) & := & \qty{ (\bm{p}, \varepsilon) \in \mathbb{R}^4 \ |\  \Phi = \pardiff{\Phi}{s}=\pardiff{\Phi}{t}=0 \text{ for some } (s,t)\in\R^2 }\  \\
                & & = & \qty{ (\bm{p}, \varepsilon) \in \mathbb{R}^4 \ |\  \left\langle \pardiff{f}{s}, f(s,t)-\bm{p} \right\rangle =  \left\langle \pardiff{f}{t}, f(s,t)-\bm{p} \right\rangle = \Phi = 0 \text{ for some } (s,t)\in\R^2} \\
                & & = & \qty{ (\bm{p}, \varepsilon) \in \mathbb{R}^4 \ |\  f(s,t) - \bm{p} = ^\exists \lambda \nu, \ || f(s,t) - \bm{p} ||^2 = \varepsilon^2 \text{ for some } (s,t)\in\R^2} \\
                & & = & {\qty{ (\bm{p}, \varepsilon) \in \mathbb{R}^4 \ |\  \bm{p} = f(s,t) - \varepsilon \bm{\nu} \text{ for some } (s,t)\in\R^2}},
            \end{array}}$$
        where $\bm{\nu}$ is the unit normal vector field of $f$. Its intersection with a hypersurface $\varepsilon = \varepsilon_0$ gives a parallel surface of $f$ at distance $\varepsilon_0$. This explains why identifying singularities of $\Phi$ is crucial.
        
        The 2-jet of $\Phi$ is
            $$j^2 \Phi = \Phi(s,t, \bm{p}, \varepsilon) = -\varepsilon^2 + p^2 +q^2 +r^2 - 2ps + (-\kappa_0q + 1)s^2 + (-q\cos\theta_0 + r\sin\theta_0)t^2,$$
        where $\bm{p} = (p,q,r)\in \R^3$. Then $\Phi$ is $\mathcal{R}$-equivalent to an $A_1$-singularity at \Origin{} if and only if $p=0$, $q\not=\frac{1}{\kappa_0}$ and $-q\cos\theta_0 + r \sin\theta_0 \not= 0$. 
    
        We denote it by a $A^t_k$ or $A^s_k$ singularity that the $\Phi$ is degenerated to a corank 1 singularity if a kernel direction of the Hessian for $\Phi$ is an $s$ direction or a $t$ direction, respectively.
        
        \begin{prop} \label{prop:AksDeg}
            Let $f$ be parameterized in Fukui's normal form and have a cuspidal cross cap at \Origin{} and let $\Phi$ be a distance squared unfolding. Then, \vspace{-0.25em}
            \begin{itemize}
                \item[\ronumi] The distance squared unfolding $\Phi$ is $\mathcal{R}$-equivalent to $A_1^{s,\pm}$ singularity, $s^2 \pm t^2$ if and only if $-q\cos\theta_0 + r\sin\theta_0 \not = 0$ and $q\not=\frac{1}{\kappa_0}$.  \vspace{-0.5em}
                \item[\ronumii] The distance squared unfolding $\Phi$ is $\mathcal{R}$-equivalent to $A_2^{s,\pm}$ singularity, $s^3 \pm t^2$ if and only if $-q\cos\theta_0 + r\sin\theta_0 \not = 0$, $q=\frac{1}{\kappa_0}$ and $r\not=-\frac{\kappa_1}{\kappa_0^2\tau_0}$. \vspace{-0.5em}
                \item[\ronumiii] The distance squared unfolding $\Phi$ is $\mathcal{R}$-equivalent to $A_3^{s,\pm}$ singularity, $s^4 \pm t^2$ if and only if $-q\cos\theta_0 + r\sin\theta_0 \not = 0$, $q=\frac{1}{\kappa_0}$, $r=-\frac{\kappa_1}{\kappa_0^2\tau_0}$ and \vspace{-0.5em}
                    $$\kappa_0^2\tau_0^3 + \kappa_0(\kappa_1\tau_1-\kappa_2\tau_0) + 2\kappa_1^2\tau_0 \not = 0.$$
            \end{itemize}
        \end{prop}
    
        \begin{rmk}
            These conditions depend only on the information of the singular curve $\gamma$ of the surface $f$, so the same conditions hold for parallel curves. Furthermore, a square of the distance $\varepsilon^2$ that the parallel surface $f^\varepsilon$ degenerates equals $p^2 + q^2 + r^2$. Hance, $\varepsilon^2 = \frac{\kappa_0^2 \tau_0^2 + \kappa_1^2}{\kappa_0^4 \tau_0^2}$ holds if $\Phi$ is $\mathcal{R}$-equivalent to $s^4 \pm t^2$ or worse.
        \end{rmk}
    
        \begin{prop}\label{prop:AktDeg}
             Let $f$ be parameterized in Fukui's normal form and have a cuspidal cross cap at \Origin{} and $\Phi$ be the distance squared unfolding. Then, \vspace{-0.25em}
             \begin{itemize}
                 \item[\ronumi] The distance squared unfolding $\Phi$ is $\mathcal{R}$-equivalent to $A_1^{t,\pm}$ singularity, $s^2 \pm t^2$, if and only if $q\not=\frac{1}{\kappa_0}$ and $-q\cos\theta_0 + r\sin\theta_0 \not = 0$.  \vspace{-0.5em}
                 \item[\ronumii] The distance squared unfolding $\Phi$ is $\mathcal{R}$-equivalent to $A_3^{t,\pm}$ singularity, $s^2 \pm t^4$, if and only if $q\not=\frac{1}{\kappa_0}$, $-q\cos\theta_0 + r\sin\theta_0 = 0$. \vspace{-0.25em}
                    \begin{equation} \label{eq:C2}
                        C_2\qty(\frac{r}{\cos\theta_0}):=\qty(\frac{r}{\cos\theta_0})^2 (\tau_0-\theta_1)^2 - \qty(\frac{b_{04}}{3}\frac{r}{\cos\theta_0} - 1)\qty(\frac{r}{\cos\theta_0} \kappa_0\sin\theta_0 -1 )\not=0, \vspace{-0.25em}
                    \end{equation}
                where $C(\varepsilon)=\varepsilon^2(\tau_0-\theta_1)^2 - \qty(\frac{b_{04}}{3}\varepsilon + 1)(\varepsilon \kappa_0\sin\theta_0 + 1)$.
                 \item[\ronumiii] The distance squared unfolding $\Phi$ is $\mathcal{R}$-equivalent to $A_4^{t,\pm}$ singularity, $s^2 \pm t^5$, if and only if $q\not=\frac{1}{\kappa_0}$, $-q\cos\theta_0 + r\sin\theta_0 = 0$, $C_2\qty(\frac{r}{\cos\theta_0}) = 0$ and  \vspace{-0.25em}
                    $$\qty{b_{13}(\tau_0-\theta_1) - \frac{b_{05}}{10}\kappa_0\sin\theta_0}\frac{r}{\cos\theta_0} + \frac{b_{05}}{10} \not= 0. \vspace{-0.25em}$$
             \end{itemize}
        \end{prop}
    
        \begin{rmk}
            The function $\Phi$ cannot be $\mathcal{R}$-equivalent to $s^2 + t^3$ if we assume $f$ is a cuspidal cross cap or worse, and $\Phi$ can be if a cuspidal edge. Because there is a diffeomorphism such that
                $$j^3(\Phi\circ\varphi)-j^2(\Phi\circ\varphi)=\frac{b_{03}}{3}(q\sin\theta_0+r\cos\theta_0)t^3.$$
            The function $\Phi$ degenerates only at $(p,q,r)=(0,0,0)$. Furthermore, the square of the distance $\varepsilon^2$ of the parallel surface $f^\varepsilon$ equals $p^2 + q^2 + r^2$. Hance, $\varepsilon = \frac{r}{\cos\theta_0}$ if $\Phi$ is $\mathcal{R}$-equivalent to $s^2 \pm t^4$ or worse. It is a solution of the quadratic equation $C_2(\varepsilon) = 0$.
        \end{rmk}
    
        \begin{prop} \label{prop:corank2Deg}
            Let $f$ be parameterized in Fukui's normal form and have a cuspidal cross cap at \Origin{} and $\Phi$ be the distance squared unfolding. Then, \vspace{-0.25em}
            \begin{itemize}
                 \item[\ronumi] The distance squared unfolding $\Phi$ is $\mathcal{R}$-equivalent to $D_4^\pm$ singularity, $s^3 \pm st^2$, if and only if $q=\frac{1}{\kappa_0}$, $-q\cos\theta_0 + r\sin\theta_0 = 0$, $\tau_0-\theta_1 \not=0$ and  \vspace{-0.25em}
                    $$\kappa_0\tau_0\cos\theta_0 + \kappa_1\sin\theta_0 \not=0. \vspace{-0.5em}$$
                 \item[\ronumii] The distance squared unfolding $\Phi$ is $\mathcal{R}$-equivalent to $E_6^\pm$ singularity, $s^3 \pm t^4$, if and only if $q=\frac{1}{\kappa_0}$, $-q\cos\theta_0 + r\sin\theta_0 = 0$, $\tau_0-\theta_1 = 0$,  \vspace{-0.25em}
                    $$\frac{b_{04}}{3} - \kappa_0 \sin\theta_0 \not= 0 \text{ and } \kappa_0\tau_0\cos\theta_0 + \kappa_1\sin\theta_0 \not=0. \vspace{-0.25em}$$
                \item[\ronumiii] The distance squared unfolding $\Phi$ is $\mathcal{R}$-equivalent to $E_7$ singularity, $s^3 + st^3$, if and only if $q=\frac{1}{\kappa_0}$, $-q\cos\theta_0 + r\sin\theta_0 = 0$, $\tau_0-\theta_1 = 0$,  \vspace{-0.25em}
                    $$\frac{b_{04}}{3} - \kappa_0 \sin\theta_0 = 0, b_{13}\not= 0 \text{ and } \kappa_0\tau_0\cos\theta_0 + \kappa_1\sin\theta_0 \not=0. \vspace{-0.25em}$$
                \item[\ronumiv] The distance squared unfolding $\Phi$ is $\mathcal{R}$-equivalent to $E_8$ singularity, $s^3 + t^5$, if and only if $q=\frac{1}{\kappa_0}$, $-q\cos\theta_0 + r\sin\theta_0 = 0$, $\tau_0-\theta_1 = 0$,  \vspace{-0.25em}
                    $$\frac{b_{04}}{3} - \kappa_0 \sin\theta_0 = 0, b_{13} = 0, b_{05} \not= 0 \text{ and } \kappa_0\tau_0\cos\theta_0 + \kappa_1\sin\theta_0 \not=0. \vspace{-0.25em}$$
             \end{itemize}
        \end{prop}
    
        \begin{rmk}
            If we assume $f$ is a cuspidal cross cap, then $\Phi$ cannot degenerate to $E_8$-singularity because of $b_{13}\not=0$, and if cuspidal $S_1$, $\Phi$ cannot degenerate to $E_7$-singularity. Because $b_{13}=0$, $b_{14}b_{05}>0$ holds. Furthermore, a square of the distance $\varepsilon^2$ of the parallel surface $f^\varepsilon$ equals $p^2 + q^2 + r^2$. Thus, if $\Phi$ has a corank 2 singularity, then $\varepsilon = \frac{1}{\kappa_0 \sin\theta_0}$ holds. We denote $\varepsilon \kappa_0\sin\theta_0+1$ by $C_1(\varepsilon)$.
        \end{rmk}
    
        We give proofs of these propositions in the section \ref{sec:Proofs}.
        
        \begin{rmk}
            In the case $f$ is a regular surface, the distance squared unfolding $\Phi$ can be $\mathcal{K}$-versal due to \cite[Theorem 3.4.]{FukuiHasegawaPS}. On the other hand, in a cuspidal cross cap case, $\Phi$ can't be a $\mathcal{K}$-versal unfolding, because the discriminant set of $\Phi$, which attains a parallel surface of, as a frontal singularity, a cuspidal cross cap is diffeomorphic to front singularities if $\Phi$ is a $\mathcal{K}$-versal unfolding of $A_k$, $D_k$ and $E_k$-singularities. 
        \end{rmk}
    
        In a general context, we collect the worst singularity of the distance squared unfolding $\Phi$ of the surface $f$ with a cuspidal cross cap. We can rewrite conditions for these singularities using invariants of cuspidal cross caps.
    
        \begin{lem}\label{lem:Curvatures}
            Several geometric invariants for cuspidal edges were already defined. Here is a list of these invariants.
            \begin{itemize}
                \item normal curvature $\kappa_\nu$ and singular curvature $\kappa_s$ in \cite{SUY},
                \item cuspidal curvature $\kappa_c$ in \cite{MSUY}, and
                \item cuspidal directional torsion $\kappa_t$ and edge inflectional curvature $\kappa_i$ in \cite{MS}.
            \end{itemize}
    
            We express constant terms of them in Fukui's form.
                $$\begin{array}{lll}
                     \kappa_\nu(0) = \kappa_0\cos\theta_0 & \kappa_s(0) = \kappa_0\sin\theta_0 & \kappa_i(0) = \kappa_0\tau_0\cos\theta_0 + \kappa_1\sin\theta_0 \\
                     \kappa_t(0) = \tau_0-\theta_0 & \kappa_c(0) = b_{03}
                \end{array}$$
        \end{lem}
    
        \begin{proof}
            Substituting the origin by invariants of Lemma\ref{lem:invariants}.
        \end{proof}
            
        \begin{rmk}
            A map $f$ has cuspidal cross cap at the origin if and only if cuspidal curvature $\kappa_c(0)=0$ and $\kappa_c'(0)\not=0$ due to Fact \ref{fact:INV}, and cuspidal $S_1$ \textbf{only if} $\kappa(0)=\kappa'(0)=0$ and $\kappa''(0)\not=0$ due to Proposition \ref{prop:DegCSk}.
        \end{rmk}
        
        \begin{cor} \label{prop:DegConSumm}
            Let $f$ be parameterized in Fukui's normal form and have a cuspidal cross cap at \Origin{}, $f^\varepsilon$ be a parallel surface of $f$ at distance $\varepsilon$, and $\Phi(s,t, \bm{p}, \varepsilon)$ be the distance squared unfolding. Then, \vspace{-0.25em}
            \begin{itemize}
                \item[\ronumi] The singular set of $\ps$ is one of constant principal curvature lines, and $\Phi$ is equivalent to $A_{2\leq}$ singularity if and only if vectors $\partial_s$ and $\partial_t$ transverse to the constant principal curvature line in the source.
                \item[\ronumii] If $\Phi$ is $\mathcal{R}$-equivalent to $s^4+t^2$, $\partial_s$ is tangent to the constant principal curvature line in the source and \vspace{-0.25em}
                    $$\kappa_0^2\tau_0^3 + \kappa_0(\kappa_1\tau_1-\kappa_2\tau_0) + 2\kappa_1^2\tau_0 \not = 0.$$
                \item[\ronumiii] If $\Phi$ is $\mathcal{R}$-equivalent to $s^2+t^5$, $\partial_t$ is tangent to the constant principal curvature line in the source and \vspace{-0.25em}
                    $$\qty(b_3'(0) -\frac{b_5(0)^2}{100}) \kappa_t(0) + b_3'(0)\frac{b_5(0)}{10} \qty(\frac{b_4(0)}{3} - \kappa_{\bm{\nu}}(0)) \not= 0 \vspace{-0.5em}.$$
                Moreover, the distance $\varepsilon$ is a solution of \vspace{-0.5em}
                    \begin{equation} \label{eq:At4con}
                        \qty(\frac{b_4(0)}{3}\varepsilon - 1) \qty(\kappa_{\bm{\nu}} \varepsilon -1) - \kappa_t(0)^2 \varepsilon^2 = 0.\vspace{-0.75em}
                    \end{equation}
                \item[\ronumiv] If $\Phi$ is $\mathcal{R}$-equivalent to $s^3+st^2$, $\partial_s$ is tangent to the constant principal curvature line in the source, $\kappa_i(0)\not=0$, and $\kappa_t(0)\not=0$. Moreover, the distance $\varepsilon$ equals a constant term of the limiting normal curvature $\kappa_{\bm{\nu}}(0)$. \vspace{-0.5em}
            \end{itemize}
        \end{cor}
    
        \begin{proof}
            $$\lambda^\varepsilon=\det|\ps_s, \ps_t, \Nor| = t\times \det|E-\varepsilon W|\times \det|f_s,f_t/t,\nu|=0$$
            $$\Leftrightarrow\ \det|\ps_s, \ps_t, \Nor| = t\times \det|E-\varepsilon W|=0$$
        \end{proof}
    
        \begin{rmk} \label{rmk:GaussMean}
            The quadratic and the linear coefficients of the Equation \eqref{eq:C2} are $\frac{b_{04}}{3}\kappa_0 \sin\theta_0 -(\tau_0-\theta_1)^2$ and $\kappa_0\theta_0 + \frac{b_{04}}{3}$, respectively. These coefficients correspond to the constant term of mean curvature and Gaussian curvature, respectively, skipping leading terms from Corollary \ref{cor:GH}.
                $$\begin{array}{rcl}
                    H&=& \frac{1}{t} \qty[ \frac{b_{4}}{4} + \frac{1}{2}\qty(\frac{b_4}{3} + \kappa\sin\theta)t + O(t^2) ],\\
                    K&=& \frac{1}{t} \qty[ \frac{b_3\kappa\sin\theta}{2} + \qty{ \frac{b_4\kappa\sin\theta}{3} - \qty(\tau-\theta')^2 - \frac{b_3^2\kappa\cos\theta}{4}} t + O(t^2) ].  
                \end{array}$$
            Therefore, solutions of the Equation \eqref{eq:C2} imply inducing principal curvatures.
    
            According to \cite{FukuiSW}, the principal curvatures $\kappa_1,\kappa_2$ admit the asymptotic expansion,
                $$\kappa_1 = a_{-1}\frac{1}{t} + a_0 + a_1t + O^2(t),\ \ \ \kappa_2 =  b_0 + b_1t + O^2(t). $$
            Then the quadric equation
                $$0=C_2(\varepsilon) = \varepsilon^2 -\frac{1}{2}\left.\qty{(\kappa_1+\kappa_2)t}'\right|_{t=0}\varepsilon + \left.(\kappa_1\kappa_2t)'\right|_{t=0} = \varepsilon^2 - (a_0+b_0)\varepsilon+(a_0b_0+a_{-1}b_1)$$
            holds. Hence, the radii of curvature of cuspidal cross caps are determined by the terms up to the second order in the asymptotic expansions of the principal curvatures $\kappa_1,\kappa_2$.
                
        \end{rmk}
    
        \begin{dfn}
            We call a parallel surface $f^\varepsilon$ has $\bm{A_k^t}$, $\bm{D_k}$ or $\bm{E_k}$\textbf{-singularity} if the distance squared unfolding $\Phi$ of the surface $f$ has $A_k^t$, $D_k$ or $E_k$-singularity, respectively. Here, an $A_k^t$-singularity means that $\Phi$ is $\mathcal{R}$-equivalent to $s^2+t^{k+1}$.
        \end{dfn}

    \subsection{Bifurcation}
        In generic cuspidal cross caps, there are three distances $\varepsilon$ for which the parallel surface $f^\varepsilon$ degenerates. The two of distances are that a parallel surface $f^\varepsilon$ degenerates to an $A_4^t$-singularity. Another distance is that a parallel surface $f^\varepsilon$ degenerates to a $D_4$-singularity. We will examine the order in which these singularities appear.
    
        Since $C_2(0)=-1<0$ and $C_2\qty(\frac{1}{\kappa_0\sin\theta_0}) = \qty(\frac{\tau_0-\theta_1}{\kappa_0\sin\theta_0})^2 > 0$, the one of distances of $A_4^t$-singularity lies between the origin and the distance of $D_4$-singularity. Thus, the order can be classified into three patterns (see Figure \ref{fig:ConfigAD}.). Let us assign names to these patterns.
    
        \begin{figure}[h!] 
            \centering
            \includegraphics[width=0.32\linewidth]{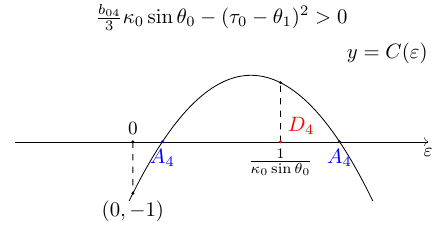}
            \includegraphics[width=0.32\linewidth]{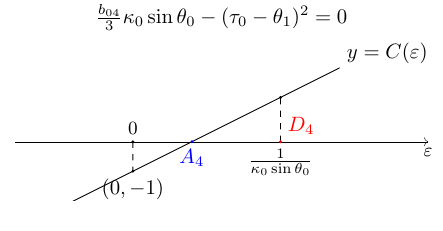}
            \includegraphics[width=0.32\linewidth]{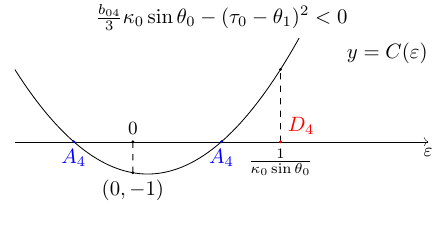}
            \caption{Configuration of $A_4^t$ and $D_4$-singularities.}
            \label{fig:ConfigAD}
        \end{figure}
    
        \begin{dfn}
            We call a cuspidal cross cap \textbf{elliptic}, \textbf{parabolic}, or \textbf{hyperbolic} if $\frac{b_{04}}{3}\kappa_0\sin\theta_0 - (\tau_0-\theta_1)^2$ is positive, zero, or negative, respectively.
        \end{dfn}
    
        Consider a case where this $A_4^t$-singularity further degenerates to an $A_5$-singularity. According to Proposition \ref{prop:AktDeg}, degeneration requires a same solution to satisfy both $C_2(\varepsilon) = 0$ and 
            \begin{equation}\label{eq:DegSolution}
                \qty{b_{13}(\tau_0-\theta_1) - \frac{b_{05}}{10}\kappa_0\sin\theta_0}\varepsilon + \frac{b_{05}}{10} = 0. \vspace{-0.25em}
            \end{equation}
        Then, we calculate these resultant and obtain the following.
            \begin{equation} \label{eq:bifA5}
                \qty(b_{13}^2 - \frac{b_{05}^2}{100}) (\tau_0-\theta_1) + b_{13} \frac{b_{05}}{10}\qty(\frac{b_{04}}{3} - \kappa_0\sin\theta_0) = 0.
            \end{equation}
        In addition, factoring $C_2$ yields the following under Equation \eqref{eq:bifA5}:
            $$C_2(\varepsilon) = \frac{1}{\frac{b_{05}}{10}b_{13}}\overbrace{\qty[\qty{b_{13}(\tau_0-\theta_1) - \frac{b_{05}}{10}\kappa_0\sin\theta_0}\varepsilon + \frac{b_{05}}{10}]}^{\text{Equation \eqref{eq:DegSolution}}} \qty[\qty{\kappa_0\sin\theta_0 b_{13} + (\tau_0-\theta_1)\frac{b_{05}}{10}} \varepsilon -b_{13}].$$
        We denote the degenerated solution by $\varepsilon_0^*$, and the non-degenerated solution by $\varepsilon_0$. A sign of difference $\varepsilon_0^*-\varepsilon_0$ indicates which singularities occur first. The sign is computed as
            $$\begin{array}{ccl}
                 \varepsilon_0^* - \varepsilon_0 & = & (\tau_0-\theta_1)(b_{13}^2 + \frac{b_{05}^2}{100})\frac{1}{b_{13}(\tau_0-\theta_1)-\frac{b_{05}}{10} \kappa_0\sin\theta_0} \frac{1}{\frac{b_{05}}{10}(\tau_0-\theta_1) + b_{13}\kappa_0\sin\theta_0}  \\
                 & = & -\frac{\tau_0-\theta_1}{b_{13} \frac{b_{05}}{10}} {\varepsilon_0\varepsilon_0^*} \qty(b_{13}^2 + \frac{b_{05}^2}{100}) \\
                 & = & \frac{\tau_0-\theta_1}{b_{13} \frac{b_{05}}{10}} \qty(\kappa_0\sin\theta_0 \frac{b_{04}}{3} - (\tau_0-\theta_1)^2) \qty(b_{13}^2 + \frac{b_{05}^2}{100}).
            \end{array}$$
        Thus, the sign is determined by the constant terms of the cuspidal directional torsion $\tau_0-\theta_1$,  our Gaussian curvature $\kappa_0\sin\theta_0\frac{b_{04}}{3} - \frac{b_{05}^2}{100}$, and the invariant of generalized cuspidal edge $b_{13}$, $b_{05}$. 
    
        Fixing $\tau_0-\theta_1$ as positive, we plot a graph of Equation (\ref{eq:bifA5}) in the $b_{13}$, $b_{05}$, $b_{14}$ space, and obtain the left side of Figure \ref{fig:BifA}.
        \begin{figure}[h!]
            \centering
            \includegraphics[width=0.32\linewidth]{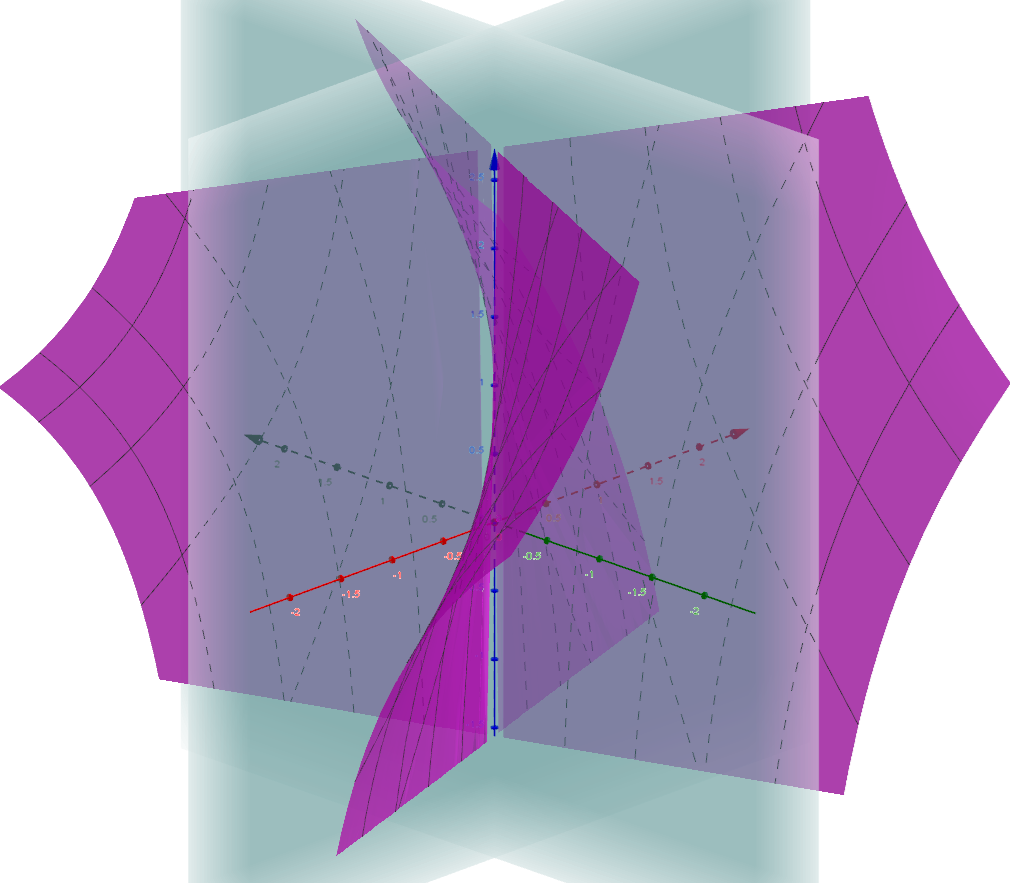}
            \includegraphics[width=0.32\linewidth]{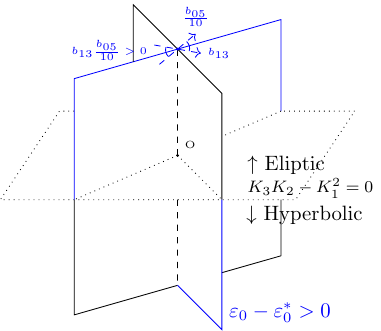}
            \includegraphics[width=0.32\linewidth]{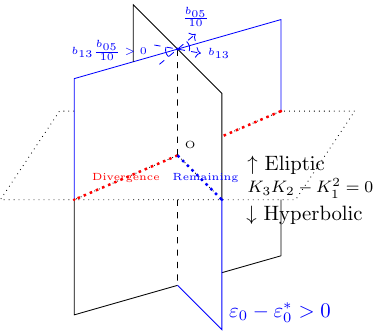}
            \caption{Left: plot of Equation (\ref{eq:bifA5}); blue planes correspond to $b_{05}=0$ or $b_{13}=0$. Center: two planes illustrate the surface defined by Equation (\ref{eq:bifA5}). Right: On the red line, an $A_5$-singularity does not appear for a parallel surface with a cuspidal cross cap; on the blue line, an $A_5$-singularity appears. Here, $K_1=\tau_0-\theta_1$, $K_2=\kappa_0\sin\theta_1$ and $K_3=\frac{b_{04}}{3}$}
            \label{fig:BifA}
        \end{figure}
        \begin{table}[h!]
            \centering
            \begin{tabular}{c@{\hskip10pt}c@{\hskip1pt}c@{\hskip1pt}c}
                 &  \textbf{Elliptic} & \textbf{Parabolic} & \textbf{Hyperbolic}  \vspace{1.5em}\\
                 
                 \hspace{-0.5em}\raisebox{5.75ex}[0ex][0ex]{$\bm{b_{13}\frac{b_{05}}{10}>0}$} & \includegraphics[scale=0.65]{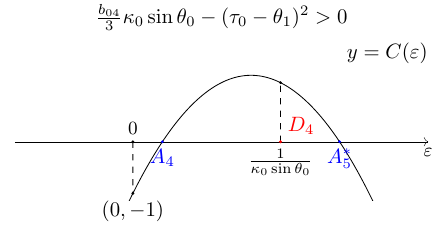} & \includegraphics[scale=0.65]{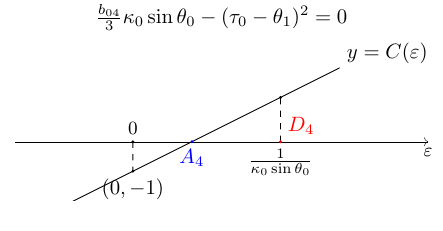} & \includegraphics[scale=0.65]{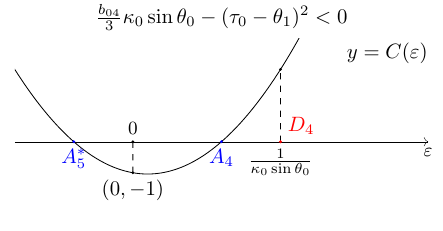}  \vspace{1em}\\
                 
                 \hspace{-0.5em}\raisebox{5.75ex}[0ex][0ex]{$\bm{b_{13}\frac{b_{05}}{10}<0}$} & \includegraphics[scale=0.65]{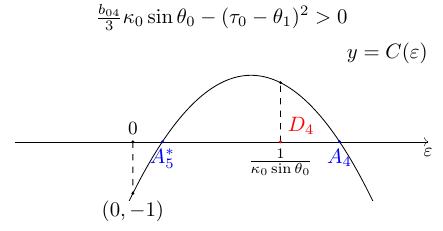} & \includegraphics[scale=0.65]{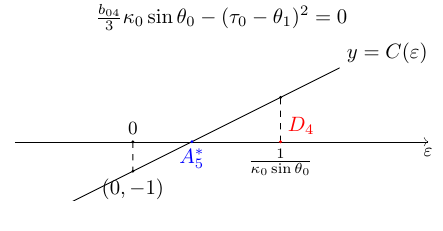} & \includegraphics[scale=0.65]{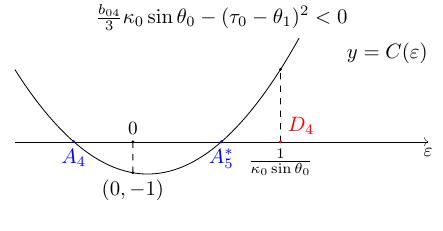}
            \end{tabular}
            \caption{Configration of $A_4^t$, $A_5$ and $D_4$ singularities.}
            \label{tab:configA5}
        \end{table}
        
        The surface defined by Equation (\ref{eq:bifA5}) is divided by two planes $b_{13}=0$ and $b_{05}=0$. For example, there are two cases where an $A_4$-singularity degenerates to an $A_5$-singularity on the left or right side if a cuspidal cross cap is an elliptic. Table \ref{tab:configA5} summarizes the configuration for $A_4^t$ and $A_5$-singularity of cuspidal cross caps. In a parabolic case, an $A_5$-singularity of cuspidal cross caps divergence if $b_{13}\frac{b_{05}}{10}>0$. Therefore, the right side of Figure \ref{fig:BifA}, representing the bifurcation set of singularities worse than $A_4^t$, includes a narrow gap of measure zero in the area $b_{13}\frac{b_{05}}{10}>0$ (See the right side of Figure \ref{fig:BifA}).

        Next, we consider the case where the $E_6$ singularity appears. A cuspidal cross cap is an $E_6$-singularity if and only if $\varepsilon = \frac{1}{\kappa_0\sin\theta_0}$, $\tau_0-\theta_1=0$, $\frac{b_{04}}{3} - \kappa_0\sin\theta_0 \not=0$ and $\kappa_i(0) \not= 0$ according to Proposition \ref{prop:corank2Deg} and Theorem \ref{prop:DegConSumm}. Under these conditions, factoring $C_2(\varepsilon)$ yields the following:
            $$C_2(\varepsilon) = - \qty(\varepsilon\frac{b_{04}}{3} - 1)\qty(\varepsilon\kappa_0\sin\theta_0 - 1).$$
        Thus, one of an $A_4^t$-singularities is located at a position as a $D_5$-singularity, and roughly speaking, they seem to combine to form an $E_6$-singularity. Hence, we write $A_4^t+D_5=E_6$. Properly speaking, combining these singularities means interpreting a singularity as an intersection between the closures of the orbits of $\Phi$'s singularities.
        
        In a elliptic case, which $A_4$-singularities combines with the $D_4$-singularity is determined by the comparison of $\frac{b_{04}}{3}$ and $\kappa_0\sin\theta_0$ (see Table \ref{tab:configE}). If $\frac{b_{04}}{3}$ is equal to $\kappa_0\sin\theta_0$, then two $A_4$-singularities combine the $D_4$-singularity, and this occurs only if $\tau_0 - \theta_1 = 0$ and $\frac{b_{04}}{3} = \kappa_0 \sin\theta_0$ since $C_2(\varepsilon)=0$ has a double root. Thus, a parallel surface of a cuspidal cross cap is an $E_7$-singularity if and only if $C_2(\varepsilon)=0$ has a double root since proposition \ref{prop:corank2Deg}.
    
    
    
                 
                 
    
    
        \begin{table}[h!] \label{tab:configE}
            \centering
            \begin{tabular}{@{\hskip0pt}ccc}
                 \textbf{Elliptic} & \textbf{Parabolic}  \\
                 
                 \includegraphics[scale=0.75]{pictures/ConfigADellip.pdf} & \includegraphics[scale=0.75]{pictures/ConfigADparab.pdf}  \\
    
                 {\footnotesize $(\kappa_0\sin\theta_0 < \frac{b_{04}}{3})$} {\Large $\downarrow$} {\footnotesize $(\kappa_0\sin\theta_0 > \frac{b_{04}}{3})$} & {\Large $\downarrow$} \vspace{0.5em} \\
                 
                 \includegraphics[scale=0.6]{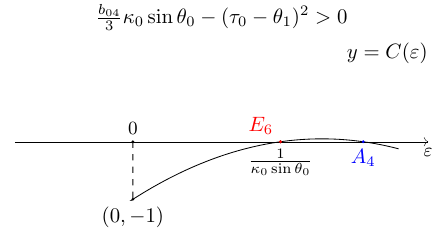} \includegraphics[scale=0.6]{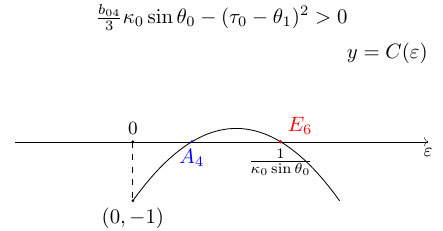} & \includegraphics[scale=0.75]{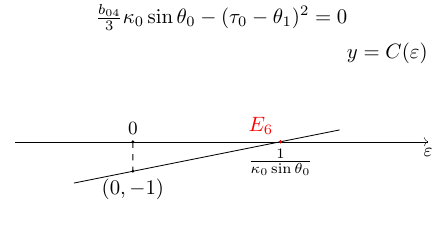} \\
                 
                 {\Large $\downarrow$} & \textbf{Hyperbolic} \vspace{0.5em}\\
    
                 \includegraphics[scale=0.75]{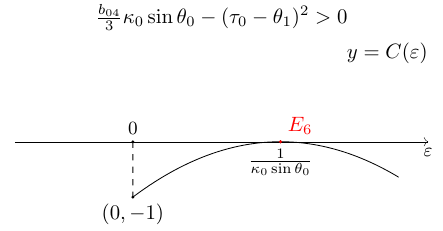} & \includegraphics[scale=0.75]{pictures/ConfigADhyper.pdf} \\
    
                 & {\Large $\downarrow$} \vspace{1em} \\
    
                 & \includegraphics[scale=0.75]{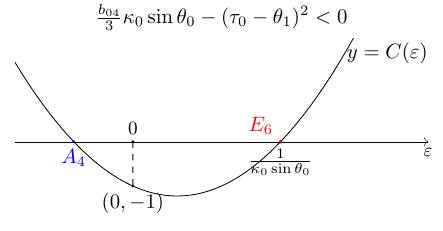}
            \end{tabular}
            \caption{Configration of $A_4^t$, $A_5$, $D_4$ and $E_6$ singularities.}
            \label{tab:my_label}
        \end{table}
    
        Finally, we discuss the adjacency of singularities.
    
        \begin{thm} \label{prop:AddSings}
            We denote combinations with singularities X and Y by $X+Y$. Then we summarize adjacency of singularities for a parallel surface of a cuspidal cross cap as follows:
                $$A_4 + D_4 = E_6, \ \ A_5 + D_4 = E_6 \text{ or } E_7, \ \ A_4 + E_6 = E_7 \ \text{ and }\  A_{\geq4} + D_{\geq5} = X_{\geq 9}.$$
            Here, $X_9$-singularity represents $\mathcal{R}$-equivalent to $x^4 + y^4 + ax^2y^2$, which is not a simple singularity. That is, adjacency between $A_4$ or worse and $D_5$ or worse yields no simple singularities. Additionally, $A_5 + D_4 = E_6$ holds if and only if $b_{05} = 0$ and $\frac{b_{04}}{3} - \kappa_0\sin\theta_0 \not=0$, whereas $A_5 + D_4 = E_7$ holds if and only if $b_{05} \not= 0$. $A_5 + E_6$ cannot happen.
        \end{thm}
        
        \begin{proof}
            The cases $A_4 + D_4 = E_6$ and $A_4 + E_6 = E_7$ are already shown. We assume there is an $A_5$-singularity, then Equation (\ref{eq:bifA5}) holds. If an $E_6$-singularity appear, $\tau_0-\theta_1=0$ and Equation (\ref{eq:bifA5}) become $b_{13}\frac{b_{05}}{10} \qty(\frac{b_{04}}{3} - \kappa_0\sin\theta_0) = 0$. Thus, $A_5 + D_4 = E_6$ if and only if $\frac{b_{05}}{10} = 0$ and $\frac{b_{04}}{3} - \kappa_0\sin\theta_0 \not=0$, and $A_5 + D_4 = E_7$ if and only if $\frac{b_{05}}{10} \not= 0$ holds. In addition, $A_5 + E_6$ cannot happen because the distance $\varepsilon_0^*$ of $A_5$-singularity is calculated from Theorem \ref{prop:corank2Deg} \ronumiii{} and we obtain the following:
                $$\varepsilon_0^* = \frac{b_{05}}{10} \frac{1}{b_{13}(\tau_0-\theta_1) - \frac{b_{05}}{10} \kappa_0 \sin\theta_0} = \frac{1}{\kappa_0\sin\theta_0}.$$
            This implies that combining $E_6$ with $A_5$ cannot occur before combining with $A_5$.
    
            Finally, we consider the case of combining $D_5$ with $A_4$. The 3-jet of $\Phi$ given by the following if a $D_5$ appears, i.e., the coefficient of $s^3$ of $\Phi$, $\kappa_0\tau_0\cos\theta_0 + \kappa_1\sin\theta_0$ equals to 0.
                $$j^3\Phi = \frac{\tau_0\cos\theta_0}{\kappa_1\sin^2\theta_0}(\tau_0-\theta_1)t^2s,$$
            and $\varepsilon=\frac{1}{\kappa_0\sin\theta_0}$. $C(\varepsilon)=0$ holds from the coincidence of the $A_4$. That is
                $$0=\varepsilon^2 (\tau_0-\theta_1)^2 - \qty(\frac{b_{04}}{3}\varepsilon - 1)\qty(\varepsilon \kappa_0\sin\theta_0 -1 ) = \varepsilon^2(\tau_0-\theta_1)^2.$$
            Consequently, we have $j^3\Phi=0$ since $\tau_0-\theta_1=0$.
        \end{proof}
    
        \begin{rmk}
            From the proof of Theorem \ref{prop:AddSings}, $A_5 + D_4 = E_6$ if and only if $\frac{b_{05}}{10} = 0$ and $\frac{b_{04}}{3} - \kappa_0\sin\theta_0 \not=0$. $A_5 + D_4 = E_7$ if and only if $\frac{b_{05}}{10} \not= 0$ or $\frac{b_{05}}{10} = \frac{b_{04}}{3} - \kappa_0\sin\theta_0 =0$.
        \end{rmk}
    
        \begin{rmk}
            We describe the adjacency of singularities for the distance squared unfolding in the cases of regular surfaces and cuspidal cross caps. For regular surfaces, the adjacency was obtained in \cite{FukuiHasegawaPS}, and is given as follows. Parallel surfaces degenerate at each principal radius of curvature.
            
                \begin{minipage}{0.48\textwidth}
                    \centering
                    \begin{tabular}{c|c}
                         Sing. of the D.S.U. & Conditions for surfaces \\ \hline\hline
                         $A_2$ &  neither ridge and umblic\\
                         $A_3$ &  1st. order ridge\\
                         $A_4$ &  2nd. order ridge\\
                         $D_4$ &  umblic\\
                    \end{tabular}        
                \end{minipage}
                \begin{minipage}{0.38\textwidth}
                    $$\xymatrix@R=10pt@C=35pt{
                        &&& D_4 \\
                        A_1 \ar[rrru]^{} \ar[r]_{} & A_2 \ar[r]_{} & A_3 \ar[r]_{} & A_4\\
                    }$$
                \end{minipage}
                
            For cuspidal cross caps, from Theorem \ref{prop:AddSings}, the adjacency is given as follows. There exists a principal radius of curvature, meaning in Remark \ref{rmk:GaussMean}, at which a parallel surface degenerates only $A_4^t$-singularity if the principal radii of curvature are distinct. This does not occur for regular surfaces.
                $$\xymatrix@R=10pt@C=20pt{
                    A_3^t+D_4 \ar[r] & 2A_4^t+D_4 \ar[rrd] \ar[r] & A_4^t+A_5^t+D_4 \ar[r] \ar[rd] \ar[rrd] & \cdots \\
                    &&& A_4^t+E_6 \ar[rd] & A_4^t+E_7 \\
                    &&&& E_7
                }$$
        \end{rmk}

    \subsection{Proof of Propsitions} \label{sec:Proofs}
        We compute the distance squared unfolding of cuspidal cross caps parameterized in Fukui's normal form. To express it, we need explicit forms of $\gamma$, $a$, $b$, $\bm{a}_2$ and $\bm{a}_3$ up to fifth degree, so we denote the coefficients of $s^it^j$ of these by $\gamma_i$, $a_{ij}$, $b_{ij}$, $\bm{a}_2^i$ and $\bm{a}_3^i$ for simply. That is;
            \begin{equation} \label{eq:coeffs}
                \begin{array}{ccccc}
                     \D a =\sum_{i,j=0}^m a_{ij} \frac{s^it^j}{i!j!} + O^{m+1}(s,t), & & b = \D \sum_{i,j=0}^m b_{ij} \frac{s^it^j}{i!j!} + O^{m+1}(s,t), \\
                     \D \bm{a}_2 =\sum_{i=0}^m (\bm{e}_1,\bm{e}_2,\bm{e}_3) \vecTri{a_2^{1i}}{a_2^{2i}}{a_2^{3i}} \frac{s^i}{i!} + O^{m+1}(s), & & \D \bm{a}_3 =\sum_{i=0}^m (\bm{e}_1,\bm{e}_2,\bm{e}_3) \vecTri{a_3^{1i}}{a_3^{2i}}{a_3^{3i}}\frac{s^i}{i!} + O^{m+1}(s), \ \text{and}\\
                     \D \gamma = \sum_{i=0}^m (\bm{e}_1,\bm{e}_2,\bm{e}_3) \vecTri{\gamma_{1i}}{\gamma_{2i}}{\gamma_{3i}} \frac{s^i}{i!} + O^{m+1}(s). & &
                \end{array}
            \end{equation}
        For example, $\gamma_3= \vecTri{-\kappa_0^2}{\kappa_1}{\kappa_0\tau_0}$, $a_{02}=1$, $a_{i3}=0$ and $\D a_{14}=-\frac{3}{2}b_{03}b_{13}$ from Equation \eqref{eq:gamma} and Lemma \ref{lem:FukuiFormCoeff}. Since
            $$\begin{array}{rcl}
                f(s,t) &=& \gamma+a\bm{a}_2+b\bm{a}_3  \\
                 &=& \D \sum_{i=0}^m \qty{ \vecTri{\gamma_{1i}}{\gamma_{2i}}{\gamma_{3i}}  + \qty(\sum_{k,l=0}^m a_{kl} \frac{s^kt^l}{k!l!}) \vecTri{a_2^{1i}}{a_2^{2i}}{a_2^{3i}} + \qty(\sum_{k,l=0}^m b_{kl} \frac{s^kt^l}{k!l!})\vecTri{a_3^{1i}}{a_3^{2i}}{a_3^{3i}}}\frac{s^i}{i!} + O^{m+1}(s,t)\\
                 &=&\D \sum_{i=0}^m \qty{ \vecTri{\gamma_{1i}}{\gamma_{2i}}{\gamma_{3i}}  + \sum_{k,l=0}^m\qty{a_{kl}  \vecTri{a_2^{1i}}{a_2^{2i}}{a_2^{3i}} + b_{kl} \vecTri{a_3^{1i}}{a_3^{2i}}{a_3^{3i}}}\frac{s^kt^l}{k!l!}}\frac{s^i}{i!} + O^{m+1}(s,t)\\
            \end{array}$$$$\begin{array}{rcl}
                 &=&\D \sum_{i=0}^m \vecTri{\gamma_{1i}}{\gamma_{2i}}{\gamma_{3i}} \frac{s^i}{i!} + \sum_{i,k,l=0}^m \qty{a_{kl}\vecTri{a_2^{1i}}{a_2^{2i}}{a_2^{3i}} + b_{kl}\vecTri{a_3^{1i}}{a_3^{2i}}{a_3^{3i}}}\frac{s^{i+k}}{i!k!}\cdot\frac{t^l}{l!} + O^{m+1}(s,t)\\
                 &=& \D \sum_{l=0}^m P_l \frac{t^l}{l!} + O^{m+1}(s,t)\ =\ \sum_{l=0}^m \vecTri{P_l^1}{P_l^2}{P_l^3} \frac{t^l}{l!} + O^{m+1}(s,t),
            \end{array}$$
        the distance squared unfolding of cuspidal cross caps is:
            \begin{equation}\label{eq:DSF}
                \begin{array}{rcl}
                     \Phi(s,t,x,y,z,\varepsilon) &=& ||(x,y,z)-f(s,t)||^2-\varepsilon^2  \\
                     &=& \D \left| \left|\vecTri{x}{y}{z} - \sum_{l=0}^m \vecTri{P_l^1}{P_l^2}{P_l^3} \frac{t^l}{l!} + O^{m+1}(s,t) \right| \right|^2 - \varepsilon^2,
                \end{array}
            \end{equation}
        where we denote $\D P_0=\sum_{i=0}^m \vecTri{\gamma_{1i}}{\gamma_{2i}}{\gamma_{3i}} \frac{s^i}{i!}$ and $\D P_l=\sum_{i,k=0}^m \qty{a_{kl}\vecTri{a_2^{1i}}{a_2^{2i}}{a_2^{3i}} + b_{kl}\vecTri{a_3^{1i}}{a_3^{2i}}{a_3^{3i}}}\frac{s^{i+k}}{i!k!} \ (l>0)$. Note that $a_{k0}\equiv b_{k0}\equiv 0$ because of the definition of Fukui's normal form. For example, by Lemma \ref{lem:FukuiFormCoeff},
            $$\begin{array}{rcccl}
                P_1 &=& \D \sum_{i,k=0}^m \qty{a_{k1}\vecTri{a_2^{1i}}{a_2^{2i}}{a_2^{3i}} + b_{k1}\vecTri{a_3^{1i}}{a_3^{2i}}{a_3^{3i}}}\frac{s^{i+k}}{i!k!} &=& 0, \\
                P_2 &=& \D \sum_{i,k=0}^m \qty{a_{k2}\vecTri{a_2^{1i}}{a_2^{2i}}{a_2^{3i}} + b_{k2}\vecTri{a_3^{1i}}{a_3^{2i}}{a_3^{3i}}}\frac{s^{i+k}}{i!k!} 
                &=& \D \sum_{i=0}^m \vecTri{a_2^{1i}}{a_2^{2i}}{a_2^{3i}}\frac{s^{i}}{i!}, \ \text{and} \\
                P_3 &=& \D \sum_{i,k=0}^m \qty{a_{k3}\vecTri{a_2^{1i}}{a_2^{2i}}{a_2^{3i}} + b_{k3}\vecTri{a_3^{1i}}{a_3^{2i}}{a_3^{3i}}}\frac{s^{i+k}}{i!k!}
                &=& \D \sum_{i,k=0}^m b_{k3}\vecTri{a_3^{1i}}{a_3^{2i}}{a_3^{3i}}\frac{s^{i+k}}{i!k!}.
            \end{array}$$
        In addition, because of Lemma \ref{lem:FukuiFormCoeff}, $a_4=-\frac{3}{4}b_3^2$, $a_5=-2b_3b_4$ in the coefficients of Equation \eqref{eq:FukuiNormalForm}, we rewrite Equations \eqref{eq:DSF} in the coefficients of Equation \eqref{eq:coeffs}.
            $$\begin{array}{l}
                \D a_4 = -\frac{3}{4}b_3^2 = -\frac{3}{4} (b_{03} + b_{13}s+\cdots)^2 = -\frac{3}{4} (b_{03}^2 + 2b_{03}b_{13}s+\cdots)  \\
                \D a_5 = -2b_3b_4  = -2(b_{03} + \cdots)(b_{04} + \cdots) = -2(b_{03}b_{04} + \cdots) 
            \end{array}$$
        Thus, $a_{04}=-\frac{3}{4}b_{03}^2$, $a_{14}=-\frac{3}{2}b_{03}b_{13}$ and $a_{05}=-2b_{03}b_{04}$. Using these coefficients, we obtain
            $$\begin{array}{rcl}
                \D j^5\qty(P_4\frac{t^4}{4!}) &=& j^5\qty(\D \sum_{i,k=0}^m \qty{a_{k4}\vecTri{a_2^{1i}}{a_2^{2i}}{a_2^{3i}} + b_{k4}\vecTri{a_3^{1i}}{a_3^{2i}}{a_3^{3i}}}\frac{s^{i+k}}{i!k!}\cdot \frac{t^4}{4!}) \\
                
                &=& \D \qty{a_{04}\vecTri{a_2^{10}}{a_2^{20}}{a_2^{30}} + b_{04}\vecTri{a_3^{10}}{a_3^{20}}{a_3^{30}}}\frac{s^{0}}{0!0!}\cdot\frac{t^4}{4!} + \qty{a_{04}\vecTri{a_2^{11}}{a_2^{21}}{a_2^{31}} + b_{04}\vecTri{a_3^{11}}{a_3^{21}}{a_3^{31}}}\frac{s^{1}}{1!0!}\cdot\frac{t^4}{4!} \\
                & &\D \ + \ \qty{a_{14}\vecTri{a_2^{10}}{a_2^{20}}{a_2^{30}} + b_{14}\vecTri{a_3^{10}}{a_3^{20}}{a_3^{30}}}\frac{s^{1}}{0!1!}\frac{t^4}{4!}\\
                &=& \D \qty{-\frac{3}{4}b_{03}^2\vecTri{a_2^{10}}{a_2^{20}}{a_2^{30}} + b_{04}\vecTri{a_3^{10}}{a_3^{20}}{a_3^{30}}}\frac{s^{0}}{0!0!}\cdot\frac{t^4}{4!} + \qty{-\frac{3}{4}b_{03}^2\vecTri{a_2^{11}}{a_2^{21}}{a_2^{31}} + b_{04}\vecTri{a_3^{11}}{a_3^{21}}{a_3^{31}}}\frac{s^{1}}{1!0!}\cdot\frac{t^4}{4!} \\
            \end{array}$$$$\begin{array}{rcl}
                & &\D \ + \ \qty{-\frac{3}{2}b_{03}b_{13} \vecTri{a_2^{10}}{a_2^{20}}{a_2^{30}} + b_{14}\vecTri{a_3^{10}}{a_3^{20}}{a_3^{30}}}\frac{s^{1}}{0!1!}\frac{t^4}{4!}, \\
            \end{array}$$
            $$\begin{array}{rcl}
                 \D j^5\qty(P_5\frac{t^5}{5!})&=& \D j^5\qty(\sum_{i,k=0}^m \qty{a_{k5}\vecTri{a_2^{1i}}{a_2^{2i}}{a_2^{3i}} + b_{k5}\vecTri{a_3^{1i}}{a_3^{2i}}{a_3^{3i}}}\frac{s^{i+k}}{i!k!}\cdot \frac{t^5}{5!})\\
                
                &=& \D \qty{-2b_{03}b_{04}\vecTri{a_2^{10}}{a_2^{20}}{a_2^{30}} + b_{05}\vecTri{a_3^{10}}{a_3^{20}}{a_3^{30}}}\frac{s^{0}}{0!0!}\cdot\frac{t^5}{5!}, \ \text{and}
            \end{array}$$
            $$j^5\qty(P_5\frac{t^5}{5!}) = j^5\qty(\sum_{i,k=0}^m \qty{a_{k5}\vecTri{a_2^{1i}}{a_2^{2i}}{a_2^{3i}} + b_{k5}\vecTri{a_3^{1i}}{a_3^{2i}}{a_3^{3i}}}\frac{s^{i+k}}{i!k!}\cdot \frac{t^5}{5!}) = \qty{-2b_{03}b_{04}\vecTri{a_2^{10}}{a_2^{20}}{a_2^{30}} + b_{05}\vecTri{a_3^{10}}{a_3^{20}}{a_3^{30}}}\frac{t^5}{5!}.$$
    
        We compute the distance squared unfolding.
            $$\begin{array}{rcl}
                \Phi-r_0 &=& \D \left| \left|\vecTri{x}{y}{z} - \sum_{l=0}^m \vecTri{P_l^1}{P_l^2}{P_l^3} \frac{t^l}{l!} + O^{m+1}(s,t) \right| \right|^2 - \varepsilon^2 -r_0\\
                &=& \D - \varepsilon^2 + \qty(x-\sum_{l=0}^mP_l^1\frac{t^l}{l!})^2 + \qty(y-\sum_{l=0}^mP_l^2\frac{t^l}{l!})^2 + \qty(z-\sum_{l=0}^mP_l^3\frac{t^l}{l!})^2 - r_0 + O^{m+1}(s,t) 
            \end{array}$$
            $$\begin{array}{cl}
                = &\D - \sum_{l=0}^m \qty(2xP_l^1 + 2yP_l^2 + 2z P_l^3)\frac{t^l}{l!} + \qty(\sum_{l=0}^mP_l^1\frac{t^l}{l!})^2 + \qty(\sum_{l=0}^mP_l^2\frac{t^l}{l!})^2 + \qty(\sum_{l=0}^mP_l^3\frac{t^l}{l!})^2  + O^{m+1}(s,t) \\
                
                = &\D - \sum_{l=0}^m \qty(2xP_l^1 + 2yP_l^2 + 2z P_l^3)\frac{t^l}{l!} + \sum_{k=0}^{m} \sum_{l=0}^k\qty(\binom{k}{l} P^1_l P^1_{k-l} + \binom{k}{l} P^2_l P^2_{k-l} + \binom{k}{l} P^3_l P^3_{k-l} )\frac{t^{k}}{k!} + O^{m+1}(s,t)\\
    
                = & \D - \sum_{k=0}^m \qty{\qty(2xP_k^1 + 2yP_k^2 + 2z P_k^3) +  \sum_{l=0}^k\binom{k}{l}\qty( P^1_l P^1_{k-l} +  P^2_l P^2_{k-l} + P^3_l P^3_{k-l} )}\frac{t^{k}}{k!} + O^{m+1}(s,t),\\
            \end{array}$$
        where $r_0=- \varepsilon^2 + x^2 + y^2 + z^2$. This is because 
            $$\qty(\sum_{l=0}^m P^i_l\frac{t^l}{l!})^2 = \sum_{k=0}^{m} \qty(\sum_{l=0}^k \binom{k}{l} P^i_l P^i_{k-l} )\frac{t^{k}}{k!} + O^{m+1}(t).$$
    
        We give Taylor expansion $\D \sum_{i=0}^m P^p_{ik}\frac{s^i}{i!} + O^{m+1}(s)$ of $P_k^p(s)$. Then,
            $$2xP_k^1 + 2yP_k^2 + 2z P_k^3 = \sum_{i=0}^m (2xP_{ik}^1 + 2yP_{ik}^2 + 2zP_{ik}^3)\frac{s^i}{i!} + O^{m+1}(s),\ \text{and}$$
            $$\begin{array}{rcl}
                \D P^p_l P^p_{k-l} &=& \D \qty(\sum_{i=0}^m P^p_{il}\frac{s^i}{i!}) \qty(\sum_{i=0}^m P^p_{i,k-l}\frac{s^i}{i!}) + O^{m+1}(s) \\
                 &=& \D \sum_{i=0}^m\sum_{j=0}^m P^p_{il} P^p_{j,k-l}\frac{s^{i+j}}{i!j!} + O^{m+1}(s)\\
                 &=& \D \sum_{n=0}^{2m}\sum_{i+j=n=0}^{2m} P^p_{il} P^p_{j,k-l}\frac{s^{n}}{i!j!} + O^{m+1}(s)\\
            \end{array}$$$$\begin{array}{rcl}
                 &=& \D \sum_{n=0}^{2m}\qty(\sum_{i+j=n=0}^{m} P^p_{il} P^p_{j,k-l}\frac{s^{n}}{i!j!} + \sum_{i+j=n=m+1}^{2m} P^p_{il} P^p_{j,k-l}\frac{s^{n}}{i!j!}) + O^{m+1}(s)\\
                 &=& \D \sum_{n=0}^{m}\qty(\sum_{i=0}^{n} P^p_{il} P^p_{n-i,k-l}\frac{s^{n}}{i!(n-i)!}) + O^{m+1}(s)\\
                 &=& \D \sum_{n=0}^{m}\qty(\sum_{i=0}^{n} \binom{n}{i}P^p_{il} P^p_{n-i,k-l})\frac{s^{n}}{n!} + O^{m+1}(s).\\
            \end{array}$$
        Thus,
            $$\begin{array}{cl}
                & \D \sum_{l=0}^k\binom{k}{l}\qty( P^1_l P^1_{k-l} +  P^2_l P^2_{k-l} + P^3_l P^3_{k-l} ) \\
    
                =& \D \sum_{l=0}^k\sum_{n=0}^{m}\sum_{i=0}^{n}\binom{k}{l}\binom{n}{i}\qty( P^1_{il} P^1_{n-i,k-l} + P^2_{il} P^2_{n-i,k-l} + P^3_{il} P^3_{n-i,k-l})\frac{s^{n}}{n!} + O^{m+1}(s)\\
    
                =& \D \sum_{n=0}^{m}\sum_{l=0}^k\sum_{i=0}^{n}\binom{k}{l}\binom{n}{i}\qty( P^1_{il} P^1_{n-i,k-l} + P^2_{il} P^2_{n-i,k-l} + P^3_{il} P^3_{n-i,k-l})\frac{s^{n}}{n!} + O^{m+1}(s).
            \end{array}$$
        Therefore, 
            $$\begin{array}{cl}
                & -\Phi+r_0 + O^{m+1}(s,t) \\
                =& \D \sum_{n,k=0}^m \qty{ (2xP_{nk}^1 + 2yP_{nk}^2 + 2zP_{nk}^3) + \sum_{l=0}^k\sum_{i=0}^{n}\binom{k}{l}\binom{n}{i}\qty( P^1_{il} P^1_{n-i,k-l} + P^2_{il} P^2_{n-i,k-l} + P^3_{il} P^3_{n-i,k-l})}\frac{s^n}{n!}\frac{t^{k}}{k!}.
            \end{array}$$
        Let $\Phi$ be $\D \sum_{i,j=0}^m \Phi_{ij} \frac{s^it^j}{i!j!} + O^{m+1}(s,t)$, then
            $$\begin{array}{rcl}
                \Phi_{00}-r_0 &=& \D -\qty{(2xP_{00}^1 + 2yP_{00}^2 + 2zP_{00}^3) + \sum_{l=0}^0\sum_{i=0}^{0}\binom{0}{l}\binom{0}{i}\qty( P^1_{il} P^1_{0-i,0-l} + P^2_{il} P^2_{0-i,0-l} + P^3_{il} P^3_{0-i,0-l})}\\
                &=& \D -\qty{(2xP_{00}^1 + 2yP_{00}^2 + 2zP_{00}^3) + \qty( P^1_{00} P^1_{00} + P^2_{00} P^2_{00} + P^3_{00} P^3_{00})}\\
                &=& 0,
            \end{array}$$
        since $P^p_{00}=\gamma_{p0} =0$ from Equation \eqref{eq:gamma}, and when $n>0$ or $k>0$,
            $$-\Phi_{nk} = (2xP_{nk}^1 + 2yP_{nk}^2 + 2zP_{nk}^3) + \sum_{l=0}^k\sum_{i=0}^{n}\binom{k}{l}\binom{n}{i}\qty( P^1_{il} P^1_{n-i,k-l} + P^2_{il} P^2_{n-i,k-l} + P^3_{il} P^3_{n-i,k-l}).$$
        We obtain $\Phi$ up to 5 order. The following are the coefficients of $\Phi$.
    
        
            $$\begin{array}{rclcrcl}
                 \Phi_{00} &=& -\varepsilon^2 + x^2 + y^2 + z^2 & &\Phi_{30} &=& \frac{1}{3}\qty(\kappa_0^2 x - \kappa_1 y + \kappa_0\tau_0 z)\\
                 \Phi_{10} &=& -2x & & \Phi_{21} &=& 0 \\
                 \Phi_{01} &=& 0 & & \Phi_{12} &=& \kappa_0\cos\theta_0x -(\tau_0-\theta_1)(\sin\theta_0 y + \cos\theta_0 z)\\ 
                 \Phi_{20} &=& 1-\kappa_0y & & \Phi_{03} &=& 0 \\
                 \Phi_{11} &=& 0 & & \\ 
                 \Phi_{02} &=& -\cos\theta_0y + \sin\theta_0 z & & \\
            \end{array}$$
    
            $$\begin{array}{rclcrcl}
                 \Phi_{40} &=& \frac{\kappa_0\kappa_1}{4}x + \frac{\kappa_0\tau_0^2-\kappa_2+\kappa_0^3}{12}y - \frac{\kappa_0\tau_1-2\kappa_1\tau_0}{12}z - \frac{\kappa_0^2}{12} & & \Phi_{41} &=& 0 \\
                 \Phi_{31} &=& 0 & & \Phi_{05} &=& -\frac{1}{60}b_{05}(\sin\theta_0 y + \cos\theta_0 z) \\ 
                 \Phi_{13} &=& -\frac{1}{3}b_{13}(\sin\theta_0 y + \cos\theta_0 z) \\ 
                 \Phi_{04} &=& -\frac{1}{12}b_{04}(\sin\theta_0 y + \cos\theta_0 z) + \frac{1}{4} \\ 
            \end{array}$$
    
            $$\begin{array}{rcl}
                \Phi_{22} &=& \frac{1}{2}(\kappa_0(\tau_0-2\theta_1)\sin\theta_0 + \kappa_1\cos\theta_0)x -\frac12((\tau_0-\theta_1)^2\sin\theta_0 + (\tau_1-\theta_2)\cos\theta_0)z \\ 
                 && \ -\ \frac12((\tau_1-\theta_2)\sin\theta_0 - ((\tau_0-\theta_1)^2+\kappa_0^2)\cos\theta_0)y -\frac12\kappa_0\cos\theta_0\\
                 \\
                 \Phi_{50} &=& -\frac{1}{60}(\kappa_0^2\tau_0^2-4\kappa_0\kappa_2-3\kappa_1^2+\kappa_0^4)x + \frac{1}{60}(3\kappa_0\tau_0\tau_1+3\kappa_1\tau_0^2-\kappa_3+6\kappa_0\kappa_1)y \\
                 && \ -\ \frac{1}{60}(\kappa_0\tau_2+3\kappa_1\tau_1-\kappa_0\tau_0^3+3\kappa_2\tau_0-\kappa_0^3\tau_0)z - \frac{1}{12}\kappa_0^2\kappa_1\\
                 \Phi_{14} &=& \frac{1}{12}(b_{04}\kappa_0\sin\theta_0x + (b_{04}(\tau_0-\theta_1)\cos\theta_0 -b_{14}\sin\theta_0)y - (b_{04}(\tau_0-\theta_1)\sin\theta_0 + b_{14}\cos\theta_0)z) \\ 
            \end{array}$$
            
            Substituting these coefficients into Proposition \ref{prop:ACriteria}, Proposition \ref{prop:DCriterion} and Proposition \ref{prop:ECriteria}, We obtain Proposition \ref{prop:AksDeg}, Proposition \ref{prop:AktDeg} and Proposition \ref{prop:corank2Deg}.

\section{Singularities of parallel surfaces} \label{Sec:4}
    \subsection{Result about the singularity of the parallel surface of the cuspidal cross cap}
         Parallel surfaces \PS of a surface $f(s,t)$ is expressed by $\ps \coloneq f + \varepsilon \Nor$. Then let the signed area density function $\Lam \coloneq \det|f^\varepsilon_s, f^\varepsilon_t, \Nor|$ of parallel surfaces $f^\varepsilon$, the locus of the singularity $\Locus$ such that $\Lam \circ \Locus =0$, and a null vector field $\Null \coloneq \eta_1^\varepsilon \partial_s + \eta_2^\varepsilon \partial_t$ provide that $\Lam$ is nondegenerate and $\ker d\ps|_{\Locus}$ is of 1-dimensional.

        The equation $\Lam=0$ can be solved with respect to $s$ or $t$ because of the implicit function theorem, allowing each variable to serve as the parameter given $\lambda^\varepsilon=0$ is nondegenerate, i.e., the coefficients of $s$ or $t$ in its linear part are non-zero from Lemma \ref{lem:LamEx}. Here, let $C_1 := \varepsilon \kappa_0\sin\theta_0 + 1$ and $C_2 := \varepsilon^2(\tau_0-\theta_1)^2 - \qty(\frac{b_{04}}{3}\varepsilon + 1)(\varepsilon \kappa_0\sin\theta_0 + 1)$ from coefficients $s,t$ of $\Lam = \frac{b_{13}\varepsilon}{2} C_1 s - C_2 t + O^2(s,t)$. 
    
        First, we assume that the coefficient of $s$ in $\Lam=0$ does not vanish, i.e., $C_1 \not= 0$. From Lemma, solving $\Lam = 0$ with respect to $s$ yields the following:
         $$s = \frac{2}{b_{13}\varepsilon}\frac{C_2}{C_1} t 
            +O^2(t).$$
        The Jacobian matrix of $f^\varepsilon$ along the locus $\gamma^\varepsilon=(s,t)$ is given by the following due to the proof of Lemma \ref{lem:EFG}.
            $$(f_s^\varepsilon,f_t^\varepsilon) \circ \gamma ^\varepsilon(t)= \qty(\begin{array}{cc}
                 \d C_1 + \frac{2C_2}{C_1b_{13}}(\tau_0-\theta_1)(\kappa_1\sin\theta_0+\kappa_0\theta_1\cos\theta_0)t + O^2(t) & \varepsilon(\tau_0-\theta_1)t + O^2(t) \\
                 \varepsilon(\tau_0-\theta_1) + \frac{1}{2C_1b_{13}}(4C_2(\tau_1-\theta_2)+\varepsilon C_1b_{13}^2)t + O^2(t) & \frac{\varepsilon^2}{C_1}(\tau_0-\theta_1)^2t + O^2(t)\\
                 O^2(t) & O^2(t)\\
            \end{array})$$
        We take a null vector field 
            $$\eta^\varepsilon\circ \gamma ^\varepsilon(t) = -\qty(\frac{\varepsilon^2}{C_1}(\tau_0-\theta_1)^2t + O^2(t))\partial_s+\qty(\varepsilon(\tau_0-\theta_1) + \frac{1}{2C_1b_{13}}(4C_2(\tau_1-\theta_2)+\varepsilon C_1b_{13}^2)t + O^2(t))\partial_t$$
        from the second row of the Jacobian matrix. This null vector field $\eta^\varepsilon$ vanishes at the origin if $\tau_0-\theta_1=0$, so we assume $\tau_0-\theta_1\not=0$.
    
        \begin{rmk}
            If we use the first row of the Jacobian matrix to obtain a null vector field $\eta^\varepsilon$, then $\eta^\varepsilon$ vanishes at the origin when $C_1 = 0$. Although this issue is irrelevant when the parameter is $t$, it is safer not to use the first row of the Jacobian matrix.
        \end{rmk}
    
         The normal vector of parallel surfaces corresponds to that of original surfaces, so we substitute the normal vector $\nu$ into the singular locus $\gamma^\varepsilon$.
            $$\nu\circ \gamma^\varepsilon = (O^2(t),O^2(t),1+O^2(t))$$
    
        The derivative of $\nu$ at $\gamma^\varepsilon$ from the definition of $\eta^\varepsilon$:
            $$\begin{array}{rcl}
                 d\nu_{\gamma^\varepsilon}(\eta^\varepsilon) &=& (\nu_s,\nu_t)\eta^\varepsilon  \\
                 &=& (f_s,f_t)\cdot W\cdot\eta^\varepsilon.
            \end{array}$$
        The singular locus $\gamma^\varepsilon$ is substituted into the Jacobi matrix of the original surface and the Weingarten matrix $W$ from Lemma \ref{lem:Weingarten}.
            $$(f_s,f_t)\circ \gamma ^\varepsilon(t) = \qty(\begin{array}{cc}
                 1 + O^2(t) & O^2(t) \\
                 O^2(t) & t + O^2(t) \\
                 O^2(t) & O^2(t)
            \end{array})$$
            \begin{multline*}
                W\circ \gamma ^\varepsilon(t) = \\
                \qty(\begin{array}{cc}
                     -\kappa_0\sin\theta_0 - \frac{2C_2}{C_1b_{13}\varepsilon}(\kappa_1\sin\theta_0+\kappa_0\theta_1\cos\theta_0)t + O^2(t) & \frac{1}{t}\qty(\theta_1-\tau_0-\frac{A}{2\varepsilon b_{13}C_1}t +O^2(t)) \\
                     (\theta_1-\tau_0)t + O^2(t) & -\frac{b_{04}\varepsilon C_1+3C_2}{3\varepsilon C_1} - \frac{B}{24C_1^2b_{13}^2\varepsilon^2}t + O^2(t),
                \end{array})
            \end{multline*}
        where $A=\varepsilon b_{13}^2C_1-4(\theta_2-\tau_1)C_2$ and $B=(3C_1^2b_{05}b_{13}^2\varepsilon^2+16C_1C_2b_{13}b_{14}\varepsilon+24C_2^2b_{23})$.  So, $d\nu_{\gamma^\varepsilon}(\eta^\varepsilon)$ is calculated as follows:
            $$d\nu_{\gamma^\varepsilon}(\eta^\varepsilon)=\qty(\begin{array}{c}
                    -\frac{\varepsilon}{C_1}(\theta_1-\tau_0)^2t + O^2(t)  \\
                    -(\theta_1-\tau_0)t + O^2(t) \\
                    O^2(t) 
            \end{array})$$
        and $\frac{d}{dt}f^\varepsilon\circ\gamma^\varepsilon$ is given by:
            $$\begin{array}{rcl}
                 \frac{d}{dt}f^\varepsilon\circ\gamma^\varepsilon &=& (f^\varepsilon_s, f^\varepsilon_t) \vecDual{\frac{ds}{dt}}{\frac{dt}{dt}} \\
                 &=& \qty(\begin{array}{cc}
                     \d C_1 + \frac{2C_2}{C_1b_{13}}K(\tau_0-\theta_1)t + O^2(t) & \varepsilon(\tau_0-\theta_1)t + O^2(t) \\
                     \varepsilon(\tau_0-\theta_1) + \frac{D}{2C_1b_{13}}t + O^2(t) & \frac{\varepsilon^2}{C_1}(\tau_0-\theta_1)^2t + O^2(t)\\
                     O^2(t) & O^2(t)\\
                \end{array}) \vecDual{\frac{2}{b_{13}\varepsilon}\frac{C_2}{C_1} 
                +O^1(t)}{1}   \\
                &=& \qty(\begin{array}{c}
                    \frac{2C_2}{\varepsilon b_{13}} + O^1(t)\\
                    \frac{2C_2}{C_1b_{13}}(\tau_0-\theta_1) + O^1(t) \\
                    O^1(t)
                \end{array}),
            \end{array}$$
        where $K=\kappa_1\sin\theta_0+\kappa_0\theta_1\cos\theta_0$, we call edge inflectional curvature in Lemma \ref{lem:Curvatures}, and $D=4C_2(\tau_1-\theta_2)+\varepsilon C_1b_{13}^2$
        We consider that $\eta$ is not transverse to $S(f)$ since
            \begin{equation} \label{eq:independent}
                \begin{split}
                    \eta\lambda(0)&=\qty(\frac{\varepsilon^2}{C_1}(\tau_0-\theta_1)^2t + O^2(t)) \cdot \qty(\frac{b_{13}\varepsilon}{2}C_1+O^1(s,t)) \\
                    &\ \ \ + \qty(\varepsilon(\tau_0-\theta_1) + \frac{1}{2C_1b_{13}}(4C_2(\tau_1-\theta_2)+\varepsilon C_1b_{13}^2) + O^2(t)) \cdot \qty(C_2 + O^1(s,t))\\
                    &= \varepsilon(\tau_0-\theta_1)C_2 + O^1(t).
                \end{split}
            \end{equation}
        
        The above arguments lead to the following theorem.
        
        \begin{thm}\label{thm:PSofCCC}
            A parallel surface of a cuspidal cross caps as written by Fukui's normal form and $\tau_0-\theta_1\not=0$ is $\mathcal{A}$-equivalent to a cuspidal cross cap if and only if $C_2(\varepsilon)\not=0$.
        \end{thm}
    
        \begin{proof}
            From \cite{SSk}, $\psi(0)=\qty|\frac{d}{dt}f^\varepsilon\circ\gamma^\varepsilon \ \nu\circ\gamma^\varepsilon\ d\nu_{\gamma^\varepsilon}(\eta^\varepsilon)|_{t=0}=0$ only if a map $f$ has a cuspidal $S_1$-singularity. Then,
                $$\psi(0)= \qty|\frac{d}{dt}f^\varepsilon\circ\gamma^\varepsilon \ \nu\circ\gamma^\varepsilon\ d\nu_{\gamma^\varepsilon}(\eta^\varepsilon)|_{t=0} = \frac{2C_2(\tau_0-\theta_1)}{C_1^2b_{13}\varepsilon}(C_1^2+\varepsilon^2(\tau_0-\theta_1)^2)t + O^2(t)$$
                
            We will discuss the case of the parameter $s$. That is solving $\lambda^\varepsilon=0$ for $t$. This necessary and sufficient condition is $C_2\not=0$ according to the implicit function theorem. For the same computation as in Theorem \ref{thm:PSofCCC},
                $$\psi(t)=\frac{b_{13}\varepsilon(\tau_0-\theta_1)}{2C_2}(C_1^2+\varepsilon^2(\tau_0-\theta_1)^2)t + 
                \frac{b_{23}\varepsilon(\tau_0-\theta_1)}{4C_2}(C_1^2+\varepsilon^2(\tau_0-\theta_1)^2)t^2 + O^3(t).$$
    
            $\eta\lambda(0)$ and $\det \qty(\gamma'\ \hat{c}'',\ 3\hat{c}^{(5)}-10lc^{(4)})(0)$ in the proof of Proposition \ref{prop:DegCSk} is not depended on coordinate. Therefore, it satisfies the conditions (a) and (b) due to Equation \eqref{eq:Cusp} and \eqref{eq:independent}
                $$\eta\lambda(0) = \varepsilon(\tau_0-\theta_1)C_2$$
                $$\det \qty(\gamma'\ \hat{c}'',\ 3\hat{c}^{(5)}-10lc^{(4)})(0)=-b_{05}\varepsilon^7(b_{04}\varepsilon-3)(\tau_0-\theta_1)^7.$$
        \end{proof}
    
        \begin{dfn}\label{def:DegCS1}
            \textbf{The degenerate cuspidal $S_1$ singularity} if the conditions in the proof of Proposition \ref{prop:DegCSk}. without (a) hold.
        \end{dfn}
    
        \begin{rmk}
            The necessary and sufficient condition needs the above calculations for 2 jet, so the computation is straightforward but rather brutal. We just state the result. A parallel surface of cuspidal cross caps is $\mathcal{A}$-equivalent to a degenerate cuspidal $S_1$ if and only if $C_2(\varepsilon)=0$, $b_{05}\not=0$ and the following hold.
                \begin{equation}\label{eq:SNCofCK1}
                    \qty{b_{13}(\tau_0-\theta_1) - \frac{b_{05}}{10}\kappa_0\sin\theta_0}\varepsilon + \frac{b_{05}}{10} \not= 0
                \end{equation}
                
            To obtain $b_{05} \neq 0$, it is sufficient to compute condition (b) in the proof of Proposition \ref{prop:DegCSk}. We set 
                $$c(x)=(kx^3+O^6(x),\varepsilon(\tau_0-\theta_1)x+O^6(x)),$$
            where $k=\frac{\varepsilon^4(\tau_0-\theta_1)^3(3b_{05}(\tau_0-\theta_1)+4(3-b_{04}\varepsilon))}{24(3\varepsilon^2(\tau_0-\theta_1)^2 - (b_{04}\varepsilon - 3)(\varepsilon \kappa_0\sin\theta_0 - 1))}$ and 
                $$l=\frac{3\varepsilon^2(\tau_0-\theta_1)(4b_{13}\varepsilon(\tau_0-\theta_1)+b_{05}(1-\varepsilon\kappa_0\sin\theta_0))}{4(3\varepsilon^2(\tau_0-\theta_1)^2 - (b_{04}\varepsilon - 3)(\varepsilon \kappa_0\sin\theta_0 - 1))},$$
            then 
                \begin{equation}\label{eq:Cusp}
                    \det \qty(\gamma'\ \hat{c}'',\ 3\hat{c}^{(5)}-10lc^{(4)})(0) = -b_{05}\varepsilon^7(b_{04}\varepsilon-3)(\tau_0-\theta_1)^7
                \end{equation}
            This is why the condition (b) if and only if $b_{05}\not=0$ since we are assuming $\tau_0-\theta_1\not=0$. 
    
            There are two distances at which the cuspidal cross cap degenerates. The more degenerate one has already been discussed in Table \ref{tab:configA5}.
        \end{rmk}
    
        \begin{rmk}
            There is a direction in which the parallel surface degenerates only deg. $cS_1$ since one solution satisfied $C_2(\varepsilon)=$ and \eqref{eq:SNCofCK1} exists. This implies parallel surfaces of a frontal have a direction degenerated to a fixed singularity. In fact, according to \cite{FHpc}, the parallel curve of a front degenerates to a fixed singularity.
        \end{rmk}
        
        \begin{rmk}
            Theorem \ref{thm:PSofCCC} corresponds to Proposition \ref{prop:AktDeg} (iv). This shows that the distance squared unfolding is a good tool to study parallel surfaces.
        \end{rmk}
        
        We have to discuss the case where $s$ is the parameter to claim correctly Theorem \ref{thm:PSofCCC}. But in this case, there is no degenerated distance $\varepsilon$.
    
        \begin{prop}\label{prop:TCon}
            A parallel surface of a cuspidal cross cap as written by Fukui's form is $\mathcal{A}$-equivalent to a cuspidal cross cap if and only if $C_2(\varepsilon)\not=0$.
        \end{prop}
    
        \begin{proof}
            We can assume that $\tau_0-\theta_1=0$, since the case $\tau_0-\theta_1\not=0$ is treated in Theorem \ref{thm:PSofCCC}. Then, $C_1\not=0$ holds because 
                $$C_2=-\qty(\frac{b04}{3}\varepsilon + 1)(\varepsilon \kappa_0\sin\theta_0 + 1) = -C_1\qty(\frac{b04}{3}\varepsilon + 1)$$
            and $\ker df^\varepsilon=\qty{0=\frac{b_{13}\varepsilon}{2}C_1s-C_2t+O^2(s,t)}$ is non singular at the origin. Thus, we obtain the singular locus as the parameter $t$ and the vector field $\eta^\varepsilon$ as follows:
                $$\begin{array}{rcl}
                    \eta^\varepsilon &=& (\varepsilon(\tau_0-\theta_1)t + O^2(t))\partial_s + \qty(C_1 + \frac{2C_2}{C_1b_{13}}(\tau_0-\theta_1)(\kappa_1\sin\theta_0+\kappa_0\theta_1\cos\theta_0)t + O^2(t))\partial_t\\
                    &=& O^2(t)\partial_s + \qty(C_1 + O^2(t))\partial_t
                \end{array}$$
            For the same way, we compute $d\nu_{\gamma^\varepsilon}\eta^\varepsilon$:
                $$d\nu_{\gamma^\varepsilon}(\eta^\varepsilon) = (f_s,f_t)\cdot W \cdot \eta^\varepsilon = \qty(\begin{array}{c}
                    (\theta_1-\tau_0)t + O^2(t) \\
                    \frac{C_1}{\varepsilon}t + O^2(t) \\
                    O^2(t)
                \end{array}) = \qty(\begin{array}{c}
                    O^2(t) \\
                    \frac{C_1}{\varepsilon}t + O^2(t) \\
                    O^2(t)
                \end{array})$$
            and $\psi$:
                $$\psi(t)=\qty|\begin{array}{ccc}
                     \frac{2C_2}{\varepsilon b_{13}} + O^1(t) & O^2(t) & O^2(t)\\
                        O^1(t) & \frac{C_1}{\varepsilon}t + O^2(t) & O^2(t)\\
                        O^1(t) & O^2(t) & 1+O^2(t)
                \end{array}| = -\frac{2C_1C_2}{b_{13}\varepsilon^2}.$$
        \end{proof}
    
        \begin{cor} \label{cor:PSofcS1}
            A parallel surface of cuspidal $S_1$s as written by Fukui's form is $\mathcal{A}$-equivalent to a cuspidal $S_1$ if and only if $C_2(\varepsilon)\not=0$.
        \end{cor}
    
        \begin{proof}
            $b_{13}=0$ if the map $f$ has a cuspidal $S_1$ singularity. Thus, a parallel surface of a cuspidal $S_1$ is not degenerated since $\lambda^\varepsilon = C_2t +O^2(s,t)$, $\psi(0)=\psi'(0)=0$ and $\psi''(0)\not=0$ in the proof of Proposition \ref{prop:TCon}
        \end{proof}
    
        

        \begin{rmk}
            If a cuspidal cross cap has $E_6$ singularity, i.e., $C_1=C_2=0$. Then, the origin is not nondegenerate point since $\Lam = \frac{b_{13}\varepsilon}{2} C_1 s + C_2 t + O^2(s,t)$
        \end{rmk}

    \subsection{Proof of Lemmata}
        \begin{lem}\label{lem:Lambda}
            Using the Weingarten matrix $W$ of a surface $M$, the signed area density function $\Lam$ can be expressed as follows:
                $$\Lam = \det|I+\varepsilon W|\cdot \det|f_s, f_t, \Nor|.$$
        \end{lem}
        \begin{proof}
            From $(\Nor_s, \Nor_t) = -(f_s, f_t)W$,
                $$\begin{array}{ccl}
                     \Lam &=& \det|\ps_s, \ps_t, \Nor|  \\
                     &=& \det|f_s+\varepsilon \Nor_s, f_t+\varepsilon \Nor_t, \Nor| \\
                     &=& \det|(f_s, f_t)(I+\varepsilon W), \Nor| \\
                     &=& \det|I+\varepsilon W|\cdot \det|f_s, f_t, \Nor|.
                \end{array}$$
            Here, the following relation is used.
                \begin{equation}\label{eq:DepMat}
                    \begin{array}{ccc}
                         \left| (v_1, v_2) \left (\begin{array}{cc}
                              a_1& b_1  \\
                              a_2& b_2
                         \end{array} \right ), v_3 \right| &=&| a_1 v_1 + a_2 v_2, b_1 v_1 + b_2 v_2, v_3| \\
                         &=& (a_1b_2 - a_2b_1)|v_1, v_2, v_3|  \\
                         &=& \left |\begin{array}{cc}
                              a_1& b_1  \\
                              a_2& b_2
                         \end{array} \right | |v_1, v_2, v_3|.  \\
                    \end{array}    
                \end{equation}
        \end{proof} 
        
        \begin{lem}\label{lem:LamEx}
            The Taylor expansion of the signed area density function $\lambda^\varepsilon$ of a parallel surface is the following.
                $$\begin{array}{rl}
                     \Lam \\
                     =& \D\varepsilon \frac{b_3}{2}\qty(1+\varepsilon\kappa\sin\theta) + \qty(-\varepsilon^2\kappa\cos\theta\frac{b_3^2}{4} + \qty(1 + \varepsilon\frac{b_4}{3})(1+\varepsilon\kappa\sin\theta) - \varepsilon^2(\tau-\theta'))t + O^2(t) \\
                     =& \D b_{03}* + \qty(\varepsilon\frac{b_{13}}{2}(1+\varepsilon\kappa_0\sin\theta_0) + b_{03}*)s + \qty(\qty(1 + \varepsilon\frac{b_{04}}{3})(1+\varepsilon\kappa_0\sin\theta_0) - \varepsilon^2(\tau_0 - \theta_1) + b_{03}^2*)t + O^2(s,t)
                \end{array}$$
            Here, the omitted coefficient $*$ is computed in the proof below.
            In particular, the singular locus of a parallel surface $\Lam=0$ is a curve through the origin if the original surface has a cuspidal cross cap at the origin, i.e., $b_{03}=0$ and the following hold in the cuspidal cross cap case.
                $$\Lam = \qty(\varepsilon\frac{b_{13}}{2}(1+\varepsilon\kappa_0\sin\theta_0))s + \qty(\qty(1 + \varepsilon\frac{b_{04}}{3})(1+\varepsilon\kappa_0\sin\theta_0) - \varepsilon^2(\tau_0 - \theta_1))t + O^2(s,t)$$
            In addition, in the case for a cuspidal $S_1$, i.e., $b_{03}=b_{13}=0$, $\Lam$ express as the following.
                $$\Lam = \qty(\qty(1 + \varepsilon\frac{b_{04}}{3})(1+\varepsilon\kappa_0\sin\theta_0) - \varepsilon^2(\tau_0 - \theta_1))t + O^2(s,t)$$
        \end{lem}
        \begin{proof}
            We aim to express a Taylor expansion of $\Lam$ in terms of the coefficients of $f$. Then, $\det|E-\varepsilon W|$ and $\det|f_s, f_t, \Nor|$ can be computed as follows, according to Lemma.\ref{lem:EFG} and Lemma. \ref{lem:Weingarten}.
                $$\begin{array}{ccl}
                    \det|I+\varepsilon W| &=& \det\matrixDD
                    {1+\varepsilon\kappa\sin\theta - \varepsilon\frac{b_3}{2}\kappa\cos\theta t + O^2(t)}
                    {\varepsilon(\tau -\theta')t + O^2(t)}
                    {\frac{1}{t}(\varepsilon(\tau-\theta') + \varepsilon b_3'\frac{t}{2} + O^2(t))}
                    {\frac{1}{t}(\varepsilon\frac{b_3}{2} + (1 + \varepsilon\frac{b_4}{3})t + O^2(t))}  \\
                    &=& \D \frac{1}{t}\qty{\varepsilon \frac{b_3}{2}\qty(1+\varepsilon\kappa\sin\theta) + \qty(-\varepsilon^2\kappa\cos\theta\frac{b_3^2}{4} + \qty(1 + \varepsilon\frac{b_4}{3})(1+\varepsilon\kappa\sin\theta) - \varepsilon^2(\tau-\theta'))t + O^2(t)} \\
                    \det|f_s, f_t, \Nor| &=& \langles{\Nor,f_s\times f_t}\ =\ |f_s\times f_t| \ =\ \sqrt{EG-F^2}\\
                    &=&  t(1 + O^2(t))
                \end{array}$$
            Therefore,
                $$\begin{array}{rcl}
                     \Lam &=& \det|E-\varepsilon W|\cdot \det|f_s, f_t, \Nor| \\
                     &=& \D\varepsilon \frac{b_3}{2}\qty(1+\varepsilon\kappa\sin\theta) + \qty(-\varepsilon^2\kappa\cos\theta\frac{b_3^2}{4} + \qty(1 + \varepsilon\frac{b_4}{3})(1+\varepsilon\kappa\sin\theta) - \varepsilon^2(\tau-\theta'))t + O^2(t) \\
                     &=& \D \varepsilon \frac{b_{03}}{2}\qty(1+\varepsilon\kappa_0\sin\theta_0) + \qty{\varepsilon \frac{b_{13}}{2}\qty(1+\varepsilon\kappa_0\sin\theta_0) + \varepsilon \frac{b_{03}}{2}\varepsilon\qty(\kappa_1\sin\theta_0+\kappa_0\theta_1\cos\theta_0)}s \\
                     && \D \hspace{3em} + \qty{-\varepsilon^2\kappa_0\cos\theta_0\frac{b_{03}^2}{4} + \qty(1 + \varepsilon\frac{b_{04}}{3})(1+\varepsilon\kappa_0\sin\theta_0) - \varepsilon^2(\tau_0-\theta_1)}t + O^2(s,t)
                \end{array}$$
        \end{proof}
    
        \begin{lem}
            The Jacobi matrix of a parallel surface of a cuspidal cross cap is the following;
                $$\begin{array}{rcl}
                     (f^\varepsilon_s, f^\varepsilon_t) &=& \qty(\begin{array}{cc}
                          1+\varepsilon\kappa\sin\theta - \varepsilon\frac{b_3}{2}\kappa\cos\theta t + O^2(t) & \varepsilon(\tau -\theta')t + O^2(t) \\
                          \varepsilon(\tau-\theta') + \varepsilon b_3'\frac{t}{2} + O^2(t) & \varepsilon\frac{b_3}{2} + (1 + \varepsilon\frac{b_4}{3})t + O^2(t) \\
                          O^2(t) & O^2(t)
                     \end{array}) \\
                     &=& \qty(\begin{array}{cc}
                          1+\varepsilon\kappa_0\sin\theta_0 +\varepsilon(\kappa_1\sin\theta_0+\kappa_0\theta_1\cos\theta_0)s + O^2(s,t) & \varepsilon(\tau_0 -\theta_1)t + O^2(s,t) \\
                          \varepsilon(\tau_0-\theta_1) + \varepsilon(\tau_1-\theta_2)s + \varepsilon b_{13}\frac{t}{2} + O^2(s,t) & \varepsilon\frac{b_{13}}{2}s + (1 + \varepsilon\frac{b_{04}}{3})t + O^2(s,t) \\
                          O^2(s,t) & O^2(s,t)
                     \end{array})
                \end{array}$$
        \end{lem}
        \begin{proof}
            Since the Weingarten formula, the following holds.
                $$\begin{array}{cl}
                     (f^\varepsilon_s,f^\varepsilon_t) \\
                     =& (f_s+\varepsilon\Nor_s,f_t+\varepsilon\Nor_t) \ =\  (f_s,f_t) + \varepsilon(\Nor_s,\Nor_t) \\
                     =& (f_s,f_t) + \varepsilon\qty(-(f_s,f_t)W) \ =\  (f_s,f_t)\qty(I +\varepsilon W)\\
                     =& \qty(\begin{array}{cc}
                          1 + O^2(t) & 0 \\
                          O^2(t) & t + O^2(t) \\
                          O(t)^2 & O^2(t)
                     \end{array})\matrixDD
                    {1+\varepsilon\kappa\sin\theta - \varepsilon\frac{b_3}{2}\kappa\cos\theta t + O^2(t)}
                    {\varepsilon(\tau -\theta')t + O^2(t)}
                    {\frac{1}{t}(\varepsilon(\tau-\theta') + \varepsilon b_3'\frac{t}{2} + O^2(t))}
                    {\frac{1}{t}(\varepsilon\frac{b_3}{2} + (1 + \varepsilon\frac{b_4}{3})t + O^2(t))}  \\
                    =& \qty(\begin{array}{cc}
                          1+\varepsilon\kappa\sin\theta - \varepsilon\frac{b_3}{2}\kappa\cos\theta t + O^2(t) & \varepsilon(\tau -\theta')t + O^2(t) \\
                          \varepsilon(\tau-\theta') + \varepsilon b_3'\frac{t}{2} + O^2(t) & \varepsilon\frac{b_3}{2} + (1 + \varepsilon\frac{b_4}{3})t + O^2(t) \\
                          O^2(t) & O^2(t)
                     \end{array})
                \end{array}$$
            Expanding this expression into a Taylor series in $s$ and $t$, we obtain
                $$(f^\varepsilon_s,f^\varepsilon_t) = \qty(\begin{array}{cc}
                          1+\varepsilon\kappa_0\sin\theta_0 +\varepsilon(\kappa_1\sin\theta_0+\kappa_0\theta_1\cos\theta_0)s + O^2(s,t) & \varepsilon(\tau_0 -\theta_1)t + O^2(s,t) \\
                          \varepsilon(\tau_0-\theta_1) + \varepsilon(\tau_1-\theta_2)s + \varepsilon b_{13}\frac{t}{2} + O^2(s,t) & \varepsilon\frac{b_{13}}{2}s + (1 + \varepsilon\frac{b_{04}}{3})t + O^2(s,t) \\
                          O^2(s,t) & O^2(s,t)
                     \end{array})$$
        \end{proof}



\end{document}